\documentclass[12pt]{amsart}
\usepackage{etex}
\usepackage{amssymb} 
\usepackage{epsf}
\usepackage[all]{xy}
\usepackage[mathscr]{eucal}

\usepackage{colordvi}
\usepackage{graphicx}

\usepackage{tabularx}
\usepackage{colortbl}

\usepackage{enumerate}
\usepackage{hyperref}
\usepackage[normalem]{ulem}	

\numberwithin{equation}{section}

\newtheorem{theorem}{Theorem}[section]

\newtheorem{proposition}[theorem]{Proposition}
\newtheorem{lemma}[theorem]{Lemma}
\newtheorem{cor}[theorem]{Corollary}
\newtheorem{ex}[theorem]{Example}

\newenvironment{example}{\begin{ex}\rm}{\end{ex}}

\newtheorem*{Hilb90}{Hilbert's Theorem 90}
\newtheorem*{Shapiro}{Shapiro's Lemma}

\theoremstyle{remark}

\theoremstyle{definition}

\newtheorem{rem}[theorem]{Remark}

\newcommand{\sfC}{{\mathsf{C}}}

\newcommand{\scrT}{{\mathscr{T}}}
\newcommand{\scrN}{{\mathscr{N}}}
\newcommand{\scrP}{{\mathscr{P}}}

\newcommand{\calG}{{\mathcal{G}}}

\newcommand{\calO}{{\mathcal{O}}}

\newcommand{\calV}{{\mathcal{V}}}

\newcommand{\frA}{{\mathfrak{A}}}
\newcommand{\frB}{{\mathfrak{B}}}
\newcommand{\frC}{{\mathfrak{C}}}

\newcommand{\A}{{\mathbb{A}}}
\newcommand{\C}{{\mathbb{C}}}
\newcommand{\G}{{\mathbb{G}}}

\renewcommand{\P}{{\mathbb{P}}}
\newcommand{\R}{{\mathbb{R}}}
\newcommand{\Rpos}{\R_{>0}}
\newcommand{\T}{{\mathbb{T}}}
\newcommand{\Z}{{\mathbb{Z}}}
\newcommand{\One}{{\mathbf{1}}}

\newcommand{\bc}{{\bf c}}

\DeclareMathOperator{\Aut}{\sf Aut}

\newcommand{\TAut}{\mbox{\sf Aut}^{\mbox{\scriptsize\sf T}}}
\DeclareMathOperator{\TAutS}{\Aut^{\sf T}_{\Sigma}}
\DeclareMathOperator{\Gal}{Gal}
\DeclareMathOperator{\Hom}{Hom}
\DeclareMathOperator{\rank}{rank}
\DeclareMathOperator{\spec}{Spec}

\DeclareMathOperator{\coker}{coker}
\DeclareMathOperator{\colim}{colim}

\DeclareMathOperator{\GL}{GL}

\DeclareMathOperator{\Img}{Im}
\DeclareMathOperator{\Ind}{Ind}
\DeclareMathOperator{\Res}{Res}
\DeclareMathOperator{\Ext}{Ext}

\newcommand{\DeCo}[1]{\Blue{#1}}
\newcommand{\str}[1]{\sout{\Blue{#1}}}
\newcommand{\demph}[1]{\DeCo{{\sl #1}}}

\newcommand{\symg}[1]{\mathfrak{S}_{#1}}
\newcommand{\cl}[1]{C\ell(#1)}

\renewcommand{\hom}[3]{\Hom_{#1}(#2, #3)}

\newcommand{\units}[1]{{#1}^{\times}}

\newcommand{\dual}[1]{{#1}^{\vee}}

\DeclareMathOperator{\Br}{Br}
\newcommand{\brauer}[1]{\Br(#1)}
\newcommand{\etakernel}{\Br_{\eta}(E/k \, | \, L/F)}

\newcommand{\gm}{\G_m}

\newcommand{\mbm}{\mathbf{m}}
\newcommand{\mbr}{\mathbf{r}}

\title{Arithmetic toric varieties}

\author[Elizondo]{E. Javier Elizondo}
\address{E. Javier Elizondo\\
     Instituto de Matem\'aticas\\
     Universidad Nacional Aut\'onoma de M\'exico\\
     \'Area de la Inv.\ cient\'\i fica\\
      Circuito Exterior, Ciudad Universitaria\\
      M\'exico, D.F.\ 04510, M\'exico}
\email{javier@math.unam.mx}
\urladdr{http://www.math.unam.mx/\~{}javier}

\author[Lima-Filho]{Paulo Lima-Filho}
\address{Paulo Lima-Filho\\
         Department of Mathematics\\
         Texas A\&M University\\
         College Station\\
         Texas \ 77843\\
         USA}
\email{plfilho@math.tamu.edu}
\urladdr{http://www.math.tamu.edu/\~{}plfilho}

\author[Sottile]{Frank Sottile}
\address{Frank Sottile\\
         Department of Mathematics\\
         Texas A\&M University\\
         College Station\\
         Texas \ 77843\\
         USA}
\email{sottile@math.tamu.edu}
\urladdr{http://www.math.tamu.edu/\~{}sottile}

\author[Teitler]{Zach Teitler}
\address{Zach Teitler\\
         Department of Mathematics\\
         Boise State University\\
         1910 University Drive\\
         Boise\\
         Idaho \ 83725--1555\\
         USA}
\email{zteitler@math.boisestate.edu}
\urladdr{http://math.boisestate.edu/\~{}zteitler}

\thanks{Research of Sottile supported in part by NSF grants DMS-070105 and DMS-1001615}
\thanks{Research of Elizondo supported in part by by CONACYT 101519  and DGAPA IN100109}

\begin{document}

\begin{abstract}
 We study toric varieties over a field $k$ that split in a Galois extension $K/k$
 using Galois cohomology with coefficients in the toric automorphism group.
 Part of this Galois cohomology fits into an exact sequence induced by the
 presentation of the class group of the toric variety.
 This perspective helps to compute the Galois cohomology, particularly for cyclic 
 Galois groups.
 We use Galois cohomology to classify $k$-forms of projective spaces when $K/k$ is cyclic,
 and we also study $k$-forms of surfaces.  
\end{abstract}

\maketitle
%

\section{Introduction}

Toric varieties provide a rich class of accessible examples in algebraic geometry.
This stems from  their {simple}
classification~\cite{Demazure,Fulton,KKMSD,Oda}:
To each fan in a lattice, there is a
normal scheme over $\Z$ equipped with a faithful action of a diagonalizable (split) torus
which has a dense orbit.
Extending scalars to a field $k$ gives the \demph{split toric variety} over $k$ associated to
the fan. 
Every normal variety over $k$ equipped with a faithful action of a split torus which has a
dense orbit is a (split) toric variety for some fan.

An \demph{arithmetic toric variety} is a normal variety $Y$ over a field $k$
that is equipped with a faithful action of a (not necessarily split) algebraic torus
$\scrT$ over $k$ which has a dense orbit in $Y$.
This dense orbit is a torsor over $\scrT$, so arithmetic toric varieties are normal
equivariant compactifications of torsors.
Extending scalars to a finite Galois extension $K/k$  over which $\scrT$ splits,
$Y_K$ becomes a split toric variety $X_\Sigma$ for some fan $\Sigma$.
Thus $Y$ is a $k$-form of the toric variety $X_{\Sigma}$.
There are non-split $k$-forms of a toric variety only when its fan has some symmetry, and
so this theory is most interesting for highly symmetric toric varieties.

The $k$-forms of a quasiprojective variety $X$ over $K$ are in bijection with the Galois
cohomology set $H^1(K/k, \Aut(X))$~(see \cite[III]{Serre} or Section~\ref{S:GalCoh}).
In general Galois cohomology classifies what are called twisted forms of $X$, and a twisted
form $Y$ descends to a variety over $k$ if and only if every $\Gal(K/k)$-orbit in $Y$ is
contained in some affine open subset.
(This condition is clearly satisfied when $X$ is quasiprojective.)

The twisted forms of the toric variety $X_{\Sigma}$ are in bijection with
the Galois cohomology set $H^1(K/k, \TAutS)$, where
 $\TAutS$ is the algebraic  group of toric automorphisms of $X_{\Sigma}$.
When $\Sigma$ is a quasiprojective fan, this classifies $k$-forms of $X_{\Sigma}$ as every
twisted form descends to a variety over $k$.
Similarly, every twisted form of $X_\Sigma$ descends to a variety over $k$ when
$K/k$ is a quadratic extension, by a result of W{\l}odarczyk~\cite{Wl93}.
For general fans $\Sigma$, we offer a simple condition which implies that a twisted
form descends to a variety over $k$.

A toric automorphism gives an automorphism of the corresponding lattice $N$
preserving the fan $\Sigma$.
Writing $\Aut_\Sigma$ for the group of such automorphisms, we have maps of algebraic groups
$\TAutS\to {\Aut_\Sigma}\hookrightarrow  \Aut(N)$ which in turn induce  maps of Galois
cohomology sets
 \begin{equation}
  \label{Eq:pi}
   H^1(K/k, \TAutS)\ \xrightarrow{\, \pi\, } \ H^1(K/k, \Aut_\Sigma)
   \xrightarrow{\, \jmath \, }  H^1(K/k, \Aut(N))\,.
 \end{equation}
The last Galois cohomology set classifies $k$-tori that split over
$K$, and the fiber of the map $\pi$ over a given element $\bc$ in $H^1(K/k, \Aut_\Sigma)$
classifies the different twisted forms of the toric variety for the torus ${_\varphi\scrT}$
associated to $\varphi=\jmath(\bc)$.
In Theorem~\ref{Th:splitting} we identify the fiber with the
quotient of the Galois cohomology set
$H^1(K/k,{_\varphi\scrT})$ (which classifies torsors over ${_\varphi\scrT}$)
by the action of $H^0(K/k,\Aut_\Sigma)$.
This leads to a classification of quasiprojective embeddings of tori extending the classical 
theory of torus embeddings.

Arithmetic toric varieties arose as tools to study anisotropic
(non-split) tori via smooth projective compactifications.
This began with Brylinski~\cite{Br79} who showed how to construct
a complete projective fan $\Sigma$ in a lattice $N$ that is invariant under the action of a
given group $G$ on $N$.
See also \cite{CTHS05}, which completed
Brylinski's construction.
Voskresenski{\u\i}~\cite{Vo82} (see also~\cite{Vo01}) started with a
torus $\scrT$ over a field $k$.
If $K$ is the splitting field of $\scrT$ then $\scrT_K\simeq\T_N$ and $\Gal(K/k)$ acts on
$N$.
Using Brylinski's $\Gal(K/k)$-invariant fan $\Sigma$, Voskresenski{\u\i} showed there is a
smooth toric variety $Y$ over $k$ with torus $\scrT$ such that $Y_K$ is isomorphic to the toric
variety $X_\Sigma$ associated to that fan.
(This is Theorem~1.3.4 in~\cite{BT95}.)
Batyrev and Tschinkel~\cite{BT95} used this to study rational points of
bounded height on compactifications of anisotropic tori.
We do not know of an attempt to classify these structures prior to Delaunay's work on
real forms of compact toric varieties~\cite{D03,D04},  in which she classifies real structures
of smooth toric surfaces.
Her work almost immediately found an application in geometric modeling when
Krasauskas~\cite{K01,KK05} proposed using Delaunay's real toric surfaces as
patches for geometric modeling.

This work of Voskresenski{\u\i} may be understood in terms of the
map $\pi$~\eqref{Eq:pi}, which
has a splitting  $H^1(K/k,\Aut_\Sigma)\hookrightarrow H^1(K/k,\TAutS)$.
When $X_\Sigma$ is smooth and projective and we have a $k$-form $\scrT$ of the torus $\T_N$
associated to a cocycle $\bc\in H^1(K/k,\Aut_\Sigma)$, the image of $\bc$ in
$H^1(K/k,\TAutS)$ corresponds to Voskresenski{\u\i}'s arithmetic toric varieties.

Huruguen recently studied~\cite{Hur} compactifications of spherical orbits,
which is both more general and more restrictive than our
work on arithmetic toric varieties.
A spherical orbit of a connected reductive algebraic group $G$ over $k$ is a pair $(X_0,x_0)$,
where $X_0$ is a homogeneous space for $G$ on which a Borel subgroup of $G$ has a dense orbit,
and $x_0\in X_0(k)$ is a $k$-rational point.
Huruguen develops an elegant theory of equivariant embeddings of spherical orbits that extends
the standard theory over algebraically closed fields, in which embeddings correspond to colored
fans~\cite{LV}.
This involves colored fans equipped with an action of the absolute Galois
group and a condition on descent.
Huruguen also gives several examples, including a three-dimensional toric variety, which do not
satisfy descent.
This is significantly more general than our work in that it applies to spherical varieties and it 
addresses the issue of descent, but it is also more restrictive in that it
requires a $k$-rational point.
This is essentially the same restriction as  in the work of Voskresenski{\u\i} and it
rules out many examples such as the Brauer-Severi varieties
of Section~\ref{S:realP1}.
\medskip

A toric variety $X_{\Sigma}$ is a geometric invariant theory quotient
of $\A^{\Sigma(1)}$, the vector space with basis the rays of $\Sigma$~\cite{Cox,Delzant}.
After possibly replacing $K$ by a field extension, any $\Gal(K/k)$-action lifts to a
permutation representation on $\A^{\Sigma(1)}$.
The class group of $X_{\Sigma}$ has an associated $\Gal(K/k)$-equivariant
presentation in which the action on the middle term is the corresponding permutation action on
$\Z^{\Sigma(1)}$. 
In Section~\ref{S:classgroup}, we show how this yields a long exact sequence
facilitating the computation of the Galois cohomology set $H^1(K/k, \TAutS)$.
We illustrate this when $\Gal(K/k)$ is a cyclic group, and use that to
classify $K/k$-forms of projective space for a cyclic extension $K$ of $k$.
In Section~\ref{sec:surf} we consider the Galois cohomology sets for fans in $\Z^2$, which
classify arithmetic toric surfaces.

In forthcoming work~\cite{RTV}, we use this classification when $k=\R$ to compute the
$\T \rtimes \Gal(\C/\R)$-equivariant cohomology of real toric varieties~and plan to use it to
investigate more refined equivariant invariants such as Bredon cohomology~\cite{Bredon}.
Similar ideas should enable the computation of $\T\rtimes \Gal(K/k)$-equivariant
cohomology of toric varieties that split over the field extension $K/k$.
We expect this perspective to be useful for arithmetic spherical varieties, extending the
work of Huruguen~\cite{Hur}.

\subsection*{Acknowledgments}
Elizondo would like to thank the hospitality and support given by the department of
mathematics at Texas A\&M University during his sabbatical year.


\section{Toric varieties, Galois cohomology, and \texorpdfstring{$k$}{k}-tori}\label{S:background}

We recall the classification and construction of toric varieties using fans and the
dual quotient construction, 
and then review Serre's treatment~\cite{Serre} of the classification of $k$-forms of a
variety and of $k$-forms of tori.
Our intention is to make this accessible to those who do not already know both the
theory of toric varieties and Galois cohomology.

Given an affine scheme $X=\spec R$ for a Noetherian ring $R$ and an ideal $I$ of $R$, we
write $\calV(I)$ for the subscheme of $X$ cut out by $I$.
For a scheme $X$ over $\Z$ and a field $K$, or for $X$ over a field $k$ and a field
extension $K$, write $X_K = X \times \spec(K)$ for the scheme obtained from $X$ by
extending scalars to $K$, and $X(K)$ for the $K$-rational points of $X$.

%
\subsection{Split toric varieties}
Demazure~\cite{Demazure} first constructed toric varieties as schemes over $\spec \Z$ from
the data of a unimodular fan.
Subsequent treatments in algebraic geometry~\cite{Fulton,KKMSD} begin with arbitrary fans,
but construct varieties over (typically algebraically closed) fields.
These latter constructions in fact give schemes over $\spec \Z$ as follows.

Let \DeCo{$N$} be a finitely generated free abelian group of rank $n$ with dual
$\DeCo{M}=\Hom(N,\Z)$.
The \demph{polar $\sigma^\vee$} of a finitely generated subsemigroup $\sigma$ of $N$
is
\[
   \DeCo{\sigma^\vee}\ :=\ \{ u\in M\mid u(v)\geq 0\quad\mbox{for all}\quad v\in\sigma\}\,,
\]
a finitely generated subsemigroup of $M$.
A \DeCo{{\sl cone}} is a finitely generated subsemigroup $\sigma$ that is
saturated,  $(\sigma^\vee)^\vee=\sigma$.
A \DeCo{{\sl face}} $\tau$ of a cone $\sigma$ is a subsemigroup of the form
\[
    \tau\ =\ \{v\in\sigma\mid u(v)=0\}
\]
for some $u\in\sigma^\vee$.
The cone $\sigma$ is \DeCo{{\sl pointed}} if $0$ is a face, in which case
$\sigma^\vee$ generates $M$.

To a pointed cone $\sigma$ in $N$, we associate the affine scheme
$\DeCo{V_\sigma}:=\spec \Z[\sigma^\vee]$ of the semigroup ring generated by
$\sigma^\vee$.
When $\tau$ is a face of $\sigma$, we have $\sigma^\vee\subset\tau^\vee$ and the
induced map $V_\tau\to V_\sigma$ is an open inclusion, as  $\Z[\tau^\vee]$ is a subring of
the quotient field of  $\Z[\sigma^\vee]$.

A \DeCo{{\sl fan} $\Sigma$} in $N$ is a finite collection of pointed cones
 in $N$ such that
 \begin{enumerate}
  \item Any face of a cone in $\Sigma$ is a cone in $\Sigma$.
  \item The intersection of any two cones of $\Sigma$ is a common face of each.
 \end{enumerate}
Given a fan $\Sigma$ in $N$ we construct the scheme \DeCo{$X_\Sigma$} by gluing the
affine schemes $V_\sigma$ for $\sigma$ a cone of $\Sigma$ along their common subschemes
corresponding to smaller cones in $\Sigma$.
Since every pointed cone $\sigma$ contains $0$ as a face,
$V_0$ is contained in $V_\sigma$ for every cone $\sigma \in \Sigma$.

Write $\T=\T_N$ for the algebraic group $\spec \Z[M]$ and let $\T(K)=\Hom(M,\units{K})$ be the
set of $K$-valued points of $\T$.
Then $\T=V_0$ and is isomorphic to $\G_m^{\rank(N)}$.
The inclusion $\Z[\sigma^\vee]\hookrightarrow \Z[M]$ of the semigroup ring into the group
ring induces a faithful action of $\T_N$ on $V_\sigma$.
That is, there is a map
\[
  \T_N \times_{\spec \Z} V_\sigma \longrightarrow V_\sigma
\]
given by the map of algebras $\Z[\sigma^\vee] \to \Z[M] \otimes \Z[\sigma^\vee]$
sending $u \mapsto u \otimes u$.
These actions are compatible with the inclusions $V_\tau\subset V_\sigma$
induced by the inclusion of a face $\tau$ of $\sigma$.
Thus $\T_N$ acts on $X_\Sigma$.
For any field $K$,
$\T_N(K)$ acts faithfully on $X_\Sigma(K)$ with a dense orbit $V_0(K)$.
Any base extension of the scheme $X_\Sigma$ is the
\DeCo{{\sl split toric variety}} associated to the fan $\Sigma$ over the given base.

Each affine scheme $V_\sigma$ for $\sigma\in\Sigma$ contains a distinguished point
$\DeCo{x_\sigma}$ corresponding to the  {prime} ideal of $\Z[\sigma^\vee]$ which is the kernel
of the map $\Z[\sigma^\vee]\to\Z$ defined by
 \begin{equation}\label{Eq:sigmacheck}
   \sigma^\vee\ \ni\ u\ \longmapsto\
   \left\{\begin{array}{rcl} 1&\quad& u\in\sigma^\perp\\
                             0&&\mbox{otherwise}\end{array}\right.\ ,
 \end{equation}
where $\sigma^\perp$ is the set of annihilators of $\sigma$ in $M$.
Note that $x_\sigma$ becomes a closed point in $V_{\sigma,K}$ after extending scalars to any field $K$
and (over $K$) the orbit $\DeCo{\calO_\sigma}$ of $x_\sigma$ is a dense
$\T_N(K)$-orbit in $V_\sigma(K)$.

Conversely, given a pair $(X, \T)$ such that $X$ is a normal variety on which the split torus
$\T$ acts faithfully with an open dense orbit, there is a lattice $N$ and a fan 
$\Sigma \subset N$ with $(X,\T) \cong (X_\Sigma, \T_N)$.
It may be recovered as described in, for example, \cite[\S2.3]{Fulton}.

When the cones of the fan $\Sigma$ span a sublattice of $N$ that does not have full rank,
the toric variety $X_\Sigma$ is the product of a torus and a smaller-dimensional toric variety
as follows.
Let $N'\subset N$ be the saturation in $N$ of the span of $\Sigma$ and write
$\Sigma'\subset N'$ for the fan $\Sigma$ considered as a fan in $N'$.
We have the split exact sequence
 \begin{equation}\label{Eq:SES}
   0\ \longrightarrow\ N'\ \longrightarrow\ N\
      \longrightarrow\ N/N'\ \longrightarrow\ 0\,,
 \end{equation}
so that $N\simeq N'\times N/N'$ and the toric variety $X_\Sigma$ likewise decomposes
\[
   X_\Sigma\ \simeq\ X_{\Sigma'}\times \T_{N/N'}\,.
\]
Since any toric automorphism of $X_{\Sigma}$ will respect this decomposition,
we will at times assume that $\Sigma$ spans a full rank sublattice of $N$.

%
%
\subsection{Automorphisms of toric varieties}

For this section, let $X=X_\Sigma$ be the split toric variety associated to a fan
$\Sigma\subset N$ with torus $\T=\T_N$.
A \DeCo{{\sl toric automorphism}} of $X_{K}$ is a pair $(\alpha,\varphi)$, where $\alpha$ is
an automorphism of the variety $X_{K}$ and $\varphi$ is a group automorphism of the torus
$\T_{K}$, and these automorphisms intertwine the action of  {$\T_{K}$} on $X_{K}$,
\[
 \xymatrix{
    \T_N\times X_K \ar[d]_{(\varphi,\alpha)} \ar[r] & X_K \ar[d]^{\alpha}\\
     \T_N\times X_K \ar[r]&X_K
 }
\]
In particular, if $t\in\T(K)$ and $x\in X(K)$ then
$ \alpha(tx) = {^{{\varphi}} t} \alpha(x)$, 
where ${^{{\varphi}} t}$ is the image of $t$ under  {$\varphi$}.
Since  $N=\Hom(\G_m,\T_N)$, any automorphism of $\T_N$ is naturally induced
by an automorphism $\varphi$ of $N$, and we use the same notation for both.

Since the fan $\Sigma$ may be recovered from the pair
$(X, \T)$, if {$(\alpha, \varphi)$} is a toric automorphism of $X_{K}$, then
$\varphi$ lies in the group $\DeCo{\Aut_\Sigma}$ of automorphisms of $N$ that
preserve the fan $\Sigma$, because it maps torus orbits to torus orbits.

Given a toric automorphism {$(\alpha, \varphi)$} of $X_{K}$, let
$\DeCo{t_\alpha}\in \T(K)$
be defined by
\[
    \alpha(x_0)\ =\ t_\alpha x_0\,.
\]
There is such a $t_\alpha$ as the orbit $\calO_0$ of $x_0$ is the unique dense orbit of
$\T(K)$ on $X_{K}$ and $\T(K)$ acts freely on {$\calO_0(K)$}.
If $(\beta,\psi)$ is another toric automorphism, then
\[
   \beta\circ\alpha(x_0)\ =\ \beta(t_\alpha x_0) \ =\
    { ^{\psi} t_\alpha} t_\beta x_0\,,
\]
and so $t_{\beta\circ\alpha}=t_\beta{^{\psi} t_\alpha}$.
Thus the map {$(\alpha,\varphi) \mapsto (t_\alpha,\varphi)$}
is a homomorphism from the group of toric automorphisms of
$X_{K}$ to the semidirect product $\T(K)\rtimes\Aut_\Sigma$.
The algebraic group $\TAutS$ is $\T \rtimes \Aut_\Sigma$ 
which has $K$-valued points $\DeCo{\TAutS(K)} = \T(K) \rtimes \Aut_\Sigma$.

%
%
\subsection{Homogeneous coordinates for toric varieties}\label{Cox}

A split toric variety $X_{\Sigma,K}$ may also be constructed as a quotient of an open subset
of affine space by an algebraic torus.
For more details and further references, see~\cite[\S2]{Cox_PSPUM}.
This construction leads to a long exact sequence that will help us to compute Galois cohomology
sets.

Let \DeCo{$\Sigma(1)$} be the set of 1-dimensional cones of $\Sigma$ which we assume spans a
full rank sublattice of $N$.
Let $\{\DeCo{v_\rho} \mid \rho \in \Sigma(1)\}$ be the standard basis for the free abelian
group $\Z {\Sigma(1)}$ and $\{\DeCo{u_\rho} \mid \rho\in\Sigma(1)\}$ be the dual basis for
$\Z^{\Sigma(1)}$, which gives coordinates for the affine space
$\A^{\Sigma(1)}:=\spec\Z[u_\rho\mid\rho\in\Sigma(1)]$. 

Every subset $\tau $ of $\Sigma(1)$ corresponds to the cone $\widetilde{\tau}$ generated by the
basis vectors $\{v_\rho\mid \rho\in\tau\}$ indexed by $\tau$.
Let $\DeCo{\widetilde{\Sigma}}$ be the fan in {$\Z {\Sigma(1)}$} whose cones are
$\widetilde{\tau}$ as $\tau$ ranges over subsets of the rays of cones $\sigma$ in the fan
$\Sigma$.
Then the split toric variety $X_{\widetilde{\Sigma}}$ is exactly $\A^{\Sigma(1)}\setminus
Z(\Sigma)$, where $Z(\Sigma)$ is the union of coordinate subspaces defined by the monomial ideal
\[
  \Big\langle \prod_{\rho\not\in\tau} u_\rho\;\mid\; \tau\subset\sigma\in\Sigma\Big\rangle
    \ =\
  \Big\langle \prod_{\rho\not\in\sigma} u_\rho\;\mid\; \sigma\in\Sigma\Big\rangle\ .
\]

To see this, recall that $X_{\widetilde{\Sigma}}$ is the union of affine varieties
$V_{\widetilde{\tau}}$.
Since each cone 
$\widetilde{\tau}$ is generated by a subset 
($\{v_\rho\mid \rho\in\tau\}$) of the coordinate vectors, we have
\[
  V_{\widetilde{\tau}}\ \simeq\
  \A^\tau\times \G_m^{\tau^c}\ =\
   \A^{\Sigma(1)}\setminus \calV\Bigl(\prod_{\rho\not\in \tau} x_\rho\Bigr)\,.
\]
(Here $\A^\tau$ is the coordinate subspace spanned by coordinate vectors $e_\rho$ indexed by
 rays $\rho\in\tau$ and $\tau^c=\Sigma(1)\setminus \tau$ is the complement of $\tau$).
Thus
\[
  X_{\widetilde{\Sigma}}\ =\
  \A^{\Sigma(1)}\setminus \bigcap_\tau \calV\Bigl(\prod_{\rho\not\in \tau} u_\rho\Bigr)
    \ =\   \A^{\Sigma(1)}\setminus
   \calV\Bigl(\Bigl\langle \prod_{\rho\not\in\tau} u_\rho\;\mid\;
        \tau\subset\sigma\in\Sigma\Bigr\rangle\Bigr)\ .
\]

Since $\Sigma$ spans a full rank sublattice of $N$, the dual of the map $\Z\Sigma(1)\to N$
gives a short exact sequence of finitely generated abelian groups
\[
  0\ \longrightarrow\  M\  \longrightarrow\  \Z^{\Sigma(1)}\
      \longrightarrow\  \cl{\Sigma}\  \longrightarrow\  0\,,
\]
where $\cl{\Sigma}$ is the class group of $X_{\Sigma}$.
See \cite[\textsection 3.4]{Fulton}.
We have the corresponding sequence of algebraic groups
\[
  1\ \longrightarrow\ G_\Sigma\ \longrightarrow\
    \G_m^{\Sigma(1)}\ \longrightarrow\ \T_N\ \longrightarrow\ 1\,,
\]
where $G_\Sigma:=\spec \Z[\cl{\Sigma}]$.
Thus we may identify $\T_N$ with $\G_m^{\Sigma(1)}/G_\Sigma$.

The homomorphism $\Z{\Sigma(1)}\to N$ induces a map $\widetilde{\Sigma}\to\Sigma$
and a surjection of toric varieties
$X_{\widetilde{\Sigma}}\to X_\Sigma$~\cite[\S1.4]{Fulton}.
This map is $\G_m^{\Sigma(1)}$-equivariant where the action on $X_\Sigma$ is through the
quotient $\T_N=\G_m^{\Sigma(1)}/G_\Sigma$.
In particular, it is $G_\Sigma$-equivariant, where $G_\Sigma$ acts trivially on $X_\Sigma$.

\begin{theorem}[Theorem 2.1 in~\cite{Cox,Cox_PSPUM}]\label{TH:Cox}
  Let $X_\Sigma$ be a toric scheme over $\spec \Z$ whose $1$-dimensional cones span a full
  rank sublattice of $N$.
  Then
 \begin{enumerate}[\rm (1)]
  \item $X_\Sigma$ is the categorical quotient $X_{\widetilde{\Sigma}}/\!/G_\Sigma$, and

  \item $X_\Sigma$ is the geometric quotient  $X_{\widetilde{\Sigma}}/G_\Sigma$
        if and only if $\Sigma$ is a simplicial fan.

 \end{enumerate}
\end{theorem}

By a categorical quotient, we mean that the map $X_{\widetilde{\Sigma}}\to X_\Sigma$ is
universal for $G_\Sigma$-equivariant maps $X_{\widetilde{\Sigma}}\to Y$, where $G_\Sigma$ acts
trivially on $Y$.
A fan $\Sigma$ is simplicial if the rays in each cone of $\Sigma$ are linearly
independent.

When $\Sigma$ does not span a full rank sublattice of $N$, we replace
$X_{\widetilde{\Sigma}}$ in Theorem~\ref{TH:Cox} by
$X_{\widetilde{\Sigma}}\times \T_{N/N'}$, where $N/N'$ comes from the exact
sequence~\eqref{Eq:SES}.

%
%
\subsection{Non-abelian cohomology}

Let $G$ be a finite group, and $A$ a group on which it acts.
If we write $\DeCo{^ga}$ for the image of $a\in A$ under $g\in G$, then
${^g(a\cdot b)}={^ga}\cdot{^gb}$.
Write $H^0(G,A)$ for the invariants, $A^G$.
A \demph{cocycle $\bc$} of $G$ in $A$ is a map $g\mapsto c_g$ of $G$ into $A$ such that
 \begin{equation}\label{Eq:cocycle}
   c_{gh}\ =\ c_g\cdot{^gc_h}\,.
 \end{equation}
This implies that $c_e=1$, where $e\in G$ and $1\in A$ are the identity elements.
Indeed, the cocycle condition~\eqref{Eq:cocycle} implies that
$c_e=c_{e^2}=c_e\cdot{^e}c_e=(c_e)^2$.

Two cocyles $\bc$ and $\bc'$ are \DeCo{{\sl cohomologous}} if there exists $b\in A$
such that $c'_g=b^{-1}\cdot c_g\cdot {^gb}$ for all $g\in G$.
This is an equivalence relation on cocycles and we write \DeCo{$H^1(G,A)$} for the set of
equivalence classes.
This \DeCo{{\sl first cohomology of $G$ {with values} in $A$}} is a pointed set having 
distinguished element the class of the unit cocycle $\DeCo{\One}$, where $\One_g:=1$, for all
$g\in G$.

When $G$ acts trivially on $A$, a cocycle is simply a group homomorphism and $H^1(G,A)$ is the
set of conjugacy classes of homomorphisms.

These cohomology sets are functorial in both $G$ and $A$, and they fit into an exact cohomology
sequence as follows.
The \DeCo{{\sl kernel}} of a map $f\colon X\to Y$ of pointed sets is $f^{-1}(y)$, where
$y$ is the distinguished element of $Y$.
Suppose $G$ acts on a group $B$, preserving a normal subgroup $A$.
Set $C=B/A$.
Then we have the sequence of pointed sets
 \begin{equation}\label{Eq:exact_sequence}
  1 \to H^0(G,A) \to H^0(G,B) \to H^0(G,C)
   \xrightarrow{\,\delta\,}
  H^1(G,A) \to H^1(G,B) \to H^1(G,C)
 \end{equation}
which is exact in that for each cohomology set, the image of the incoming map is the kernel of
the outgoing map.
The connecting homomorphism $\delta$ is defined as follows.
If $c\in H^0(G,C)=C^G$, then we choose $b\in B$ with $c=bA$.
Since $c\in C^G$, if $g\in G$, then $\DeCo{c_g}:=b^{-1}\cdot {^gb}\in A$, and this defines a
cocycle of $G$ in $A$.

When $A$ is abelian, $H^1(G,A)$ is the usual group cohomology, and the exact
sequence~\eqref{Eq:exact_sequence} may be continued with a connecting homomorphism
$\delta\colon H^1(G,C) \to H^2(G,A)$.

We will freely use the following fundamental result in group cohomology.\medskip

\begin{Shapiro}[{\cite[I.2.5]{Serre}}]
Let $H$ be a subgroup of $G$ of finite index and $A$ an abelian group on which $H$ acts.
 Then, for any $i$,
\[
   H^i(G,\Ind^G_H A)\ =\ H^i(H,A)\,.
\]
\end{Shapiro}

We  use the following property of $\Z/2\Z$ group cohomology.
Let $A$ be an abelian group on which $\Z/2\Z$ acts.
Let $\xi \cong \Z$ be the alternating $\Z[\Z/2\Z]$-module.
Then for any $n \geq 1$,
 \begin{equation}\label{Eqn:periodicity}
  H^n(\Z/2\Z , A)\ =\ H^{n+1}(\Z/2\Z , \xi \otimes A) \, .
 \end{equation}
Indeed, writing $\Z/2\Z = \{e,g\}$, we have the
free resolution~\cite[\S~2.1]{Evens} of the trivial module,
 \begin{equation}\label{Eq:Free}
  \dotsb\ \xrightarrow{e-g}\ \Z[\Z/2\Z]\ \xrightarrow{e+g}\ \Z[\Z/2\Z]\
     \xrightarrow{e-g}\ \Z[\Z/2\Z]\ \to\ \Z\ \to\ 0\, ,
 \end{equation}
and tensoring with  $\xi$ interchanges $e+g$ and $e-g$, shifting this sequence
one position. 
Applying 
$\Hom(\underline{\ },\xi\otimes A)\simeq \Hom(\underline{\ }\otimes\xi,A)$
to~\eqref{Eq:Free}
gives~\eqref{Eqn:periodicity}.

%
%
\subsection{Galois cohomology and \texorpdfstring{$k$}{k}-forms of a variety}\label{S:GalCoh}

Let $K$ be a finite Galois extension of a field $k$ with Galois group $\calG$.
For a $\calG$-group $A$, write $H^1(K/k,A)$ for the cohomology set $H^1(\calG,A)$,
the \DeCo{{\sl Galois cohomology set}} of $K/k$ with coefficients in $A$.

Suppose that $X$ and $X'$ are varieties over $k$ which become isomorphic over $K$,
$X_K \simeq  X'_K$.
We say that $X'$ is a \DeCo{{\sl $k$-form}} of the variety $X_K$.
Write \DeCo{$E(K/k,X)$} for the set of isomorphism classes of $k$-forms of $X_K$.
Under {suitable descent} assumptions,
this is in natural bijection with the Galois cohomology set {$H^1(K/k,\Aut_K(X_K))$}
with coefficients in the group \DeCo{$\Aut_K(X_K)$} of $K$-automorphisms of $X_K$.

Indeed, there is a straightforward construction (see~\cite[III.1]{Serre}) of a map
 \begin{equation}\label{Eq:theta}
   \DeCo{\theta}\ \colon\ E(K/k,X)\ \longrightarrow\ H^1(K/k,\Aut_K(X_K))\,,
 \end{equation}
with the following property.

\begin{proposition}\label{P:GD}
  The map $\theta$ is injective.
  It is bijective if $X_K$ is quasiprojective.
\end{proposition}
A proof of this proposition  is given
in~\cite[III.\S1, Proposition~5]{Serre}. 
The quasiprojective hypothesis is sufficient, but not necessary for surjectivity. 
We discuss this further.

A \DeCo{{\sl twisted action}} of $\calG$ on the $K$-variety $X_K$ (or \DeCo{{\sl twisted form}} of $X_K$)
is a group homomorphism $\rho\colon \calG\to\Aut_k(X_K)$ that covers the action
of $\calG$ on $\spec K$.
That is, for every $g\in \calG$, the diagram
\[
 \xymatrix{
    X_K\ar[d] \ar[r]^{\rho(g)} & X_K \ar[d]\\
    \spec K\ar[r]^{g}&\spec K
 }
\]
commutes
(where $g : \spec K \to \spec K$ is given by $g^{-1} : K \to K$).
In fact, the construction of the map $\theta$~\eqref{Eq:theta} makes sense for twisted
forms of $X_K$ and gives a bijection between the Galois cohomology set
$H^1(K/k,\Aut_K(X_K))$ and 
the set of isomorphism classes of twisted forms of $X_K$.
To see this, let $\bc$ be a cocycle of $\calG$ in $\Aut_K(X_K)$.
This leads to a twisted action of $\calG$ on $X_K$:
let $g\in \calG$ act on $X_K$ by $c_g\cdot g$.
Then
\[
   (c_g\cdot g)\cdot(c_h\cdot h)\ =\
   c_g \cdot g \cdot c_h \cdot g^{-1}\cdot g\cdot h
  \ =\  c_g\cdot {^gc_h}\cdot gh\ =\
   c_{gh}\cdot gh\,,
\]
which shows that this defines a $K$-linear action of $\calG$ on $X_K$.

When $Y$ is a $k$-form of $X_K$, the $\calG$-action on $Y_K$ makes it a twisted form of
$X_K$. 
A twisted form of $X_K$ lies in the image of $\theta$~\eqref{Eq:theta}
if and only if it arises from such a $k$-form $Y$,
that is if and only if it \DeCo{{\sl descends}} to a variety $Y$ over $k$.
Weil's notion of Galois descent, or more generally Grothendieck's faithfully
flat descent explains when this occurs.

\begin{proposition}\label{P:GalDes}
  A twisted form $X$ of $X_K$ descends to a variety $Y$ over $k$ if and only if every
  $\calG$-orbit in $X$ is contained in an affine open subset of $X_K$.
\end{proposition}

A proof may be found in~\cite[\S6.2]{BLR90}.
Proposition~\ref{P:GalDes} gives necessary and sufficient conditions for Galois cohomology
to classify $k$-forms of a given variety $X_K$.

%
%
\subsection{Norm homomorphisms}\label{section: norm brauer}
For each subfield $L\subset K$, let \DeCo{$\brauer{L|K}$} be the set of similarity classes of
central simple $L$-algebras that split over $K$.
These partial Brauer groups assemble into a directed system to define the Brauer group 
$\brauer{L} = \lim_{K \subset \overline{L}} \brauer{L|K}$.
When $K$ is a Galois extension of $L$, 
$\brauer{L|K}=H^2(\Gal(K/L),\units{K})$ \cite[Thm.~6.3.4]{NSW}.

Let $\calG = \Gal(K/k)$.
The \DeCo{{\sl norm}} homomorphism 
$ 
  \scrN_T\ \colon\ T(K)\ \longrightarrow T(k)
$ 
for an abelian algebraic $k$-group $T$ is defined by 
$ 
  \lambda\  \longmapsto\ \prod_{g\in \calG}\ {\rule{0pt}{11pt}^g}\!\lambda\,.
$ 
Hence,  the usual norm homomorphism $N_{K/k}$ coincides with $\scrN_{\G_m}$.
In the case where $\calG$ is cyclic, a standard computation 
(for example,~\cite[Thm.~6.2.2]{We}) shows that
 \begin{equation}\label{Eq:H^2}
   H^2(K/k, T(K) )\ =\ T(k)/\Img \scrN_T\; \, .
 \end{equation}
For $\C/\R$ the norm becomes $z\mapsto|z|^2$, so
$\Img \scrN_{\units{\C}} = \Rpos$, the positive real numbers.
Therefore
\[
  \brauer{\R}\ =\ \brauer{\R | \C}\ =\ \units{\R} / \scrN(\units{\C})\ =\ 
 \units{\R} / \Rpos\ \cong\ \{-1 , 1 \} .
\]

Perhaps the most fundamental result in Galois cohomology is due to Hilbert.

\begin{Hilb90}
Let $K/k$ be a finite Galois extension of fields.
Then \[ H^1(\Gal(K/k) , \units{K}) = 1 . \]
\end{Hilb90}

When $K/k$ is a finite cyclic extension, this is equivalent to the following statement:
if $a \in \units{K}$ has unit norm, $N_{K/k}(a)=1$, then there exists 
$b \in \units{K}$ with $a = b / {^\rho b}$,
where $\rho$ is a generator of $\Gal(K/k)$.
 To see the equivalence of the two statements, 
Let $\rho\in\Gal(K/k)$ be a generator of order $d$ and observe that by~\eqref{Eq:cocycle} 
 a cocycle $\bc\colon \Gal(K/k)\to\units{K}$ is determined by its value $c_\rho$ at $\rho$.
 Not every element of $\units{K}$ may be a value $c_\rho$, for
\[
  1\ =\ c_{\rho^d}\ =\ c_\rho\cdot {^\rho}c_\rho\;\dotsb\; {^{\rho^{d-1}}}\!c_\rho
  \ =\ N_{K/k}(c_\rho)\,.
\]
 Now observe that 
 $c_\rho=b/{^\rho}b$
 if and only if $b^{-1}\cdot c_\rho\cdot {^\rho}b=1$, so that $\bc$ is
 cohomologous to $\One$.

\subsection{Arithmetic tori}
Let $K$ be a Galois extension of a field $k$ and $N$ be a finitely generated free abelian
group. 
A \DeCo{{\sl torus over $k$ of rank $n$}} is an algebraic group $\scrT$ over $k$ such that
for some (finite) extension $K / k$ and lattice $N$ of rank $n$,
$\scrT_K \cong \T_{N,K}$.
That is, $\scrT$ is a $k$-form of the split torus $\T_{N,K}$.
As $\T_N$ is affine, the set of such $k$-forms is in natural bijection with the
Galois cohomology set $H^1(K/k,\Aut(\, \T_N) )$, by Proposition~\ref{P:GD}.

We describe the twisted action of the Galois group $\calG$ explicitly.
For $a\in\Aut(N)$ we will also write $a$ for its adjoint in $\Aut(M)$.
Let $\varphi\colon\calG\to\Aut(\, \T_N)=\Aut(N)$.
Given $g \in \calG$ and $t\in\T_N(K)=\Hom (M,\units{K})$, define
$\DeCo{{^{g_\varphi} t}} \colon M \to \units{K}$  by the composition
 \begin{equation}
  \label{eq:varphi}
  {^{g_\varphi} t} \ \colon\ M\ \xrightarrow{\ \varphi_g\ }\
                             M\ \xrightarrow{\  t \ }\  \units{K} \
                                 \xrightarrow{\ g\ }\ \units{K} \, .
 \end{equation}
Since the Galois group $\calG$ acts trivially on $\Aut(N)$, we have the following classification.

\begin{proposition}\label{P:tori}
 The $k$-forms of the torus $\T_{N,K}$ are given by conjugacy classes of homomorphisms
 $\varphi\colon \calG\to \Aut(N)$.
 The corresponding torus ${_\varphi\scrT}$ satisfies
 ${_\varphi\scrT}(K)=\T_N(K)$, but the Galois action of
 $g\in \calG$ sends $t \colon M \to \units{K}$ in $\T_N(K)$ to
 ${^{g_\varphi} t} := g\circ t \circ \varphi_g \colon M \to \units{K}$.
\end{proposition}

Equivalently, $\T_N(K)=N\otimes_\Z\units{K}$, and the twisted $\calG$-action is simply the
diagonal $\Z[\calG]$-action where $\varphi$ gives $N$ the structure of a $\Z[\calG]$-module.

\section{Arithmetic toric varieties}\label{S:Theorem}

An \DeCo{{\sl arithmetic toric variety}} over a field $k$ is a pair $(Y,\scrT)$, where $\scrT$
is a torus over $k$ and $Y$ is a normal variety over $k$ equipped with a faithful action of
$\scrT$ which has a dense orbit.
Let $K$ be a Galois extension of $k$ over which the torus $\scrT$ splits, so that
$\scrT_K\simeq\T_{N,K}$, where $N$  is the lattice of one-parameter subgroups of $\scrT$. 
According to Proposition~\ref{P:tori}, there is a conjugacy class of group homomorphisms
 \begin{equation}\label{Eq:map}
  \varphi\ \colon\ \DeCo{\calG}\ :=\ \Gal(K/k)\ \longrightarrow\ \Aut(N)
 \end{equation}
such that $\scrT={_\varphi\scrT}$.
Then $Y_K$ is a normal variety over $K$ that is equipped with a faithful action of the split
torus $\scrT_K$ which has a dense orbit and  thus $Y_K$ is isomorphic to a
split toric variety $X_{\Sigma,K}$, for some fan $\Sigma\subset N$.

Thus we have an isomorphism of pairs
\begin{equation}\label{Eq:IsoPairs}
   \psi\ \colon\ (Y_K,\scrT_K)\ \xrightarrow{\ \sim\ }\
      ( X_{\Sigma,K},{\T}_{N,K} )\ .
\end{equation}
We may use this to transfer the $\calG$-action from $(Y_K,\scrT_K)$ to
 $(X_{\Sigma,K}, \T_{N,K})$ to obtain a twisted form of
 $(X_{\Sigma,K}, \T_{N,K})$.
Since $\calG$ acts on the pair $(X_{\Sigma,K}, \T_{N,K})$, it acts on the fan
$\Sigma\subset N$, and thus the homomorphism $\varphi$~\eqref{Eq:map} for which
$\scrT={_\varphi\scrT}$ may be chosen so that $\varphi(\calG)\subset\Aut_\Sigma$.
For $g\in \calG$, define $\DeCo{t_g}\in\T_N(K)$ by
\[
   {^g x_0} = t_gx_0 \, ,
\]
where $x_0$ is the distinguished point of $X_{\Sigma,K}$ corresponding to $0\in\Sigma$.

\begin{lemma}
  The map $\calG\ni g\mapsto (t_g,\varphi_g)\in {\TAutS(K)}$ is a cocycle in
  $H^1(K/k, \TAutS)$, and the corresponding twisted form
  of $(X_{\Sigma,K}, {\T_{N,K}})$ is induced by the map
  $\psi~\eqref{Eq:IsoPairs}$.
\end{lemma}

The same formalism as in Section~\ref{S:GalCoh} gives the following
classification.

\begin{theorem}\label{Th:classification}
 Let $K$ be a Galois extension of a field $k$ with Galois group $\calG$ and $\Sigma$ a fan in the
 lattice $N$.
 \begin{enumerate}[\rm (1)]
  \item The Galois cohomology set $H^1(K/k,\TAutS)$ is in natural bijection with the
        set of twisted forms of the split toric variety {$(X_{\Sigma,K},\T_{N,K})$}, with the
        distinguished unit cocycle $\One$ corresponding to {$(X_{\Sigma,K},\T_{N,K})$}.

  \item This Galois cohomology set is in bijection with the set $E(K/k,X_\Sigma)$ of $k$-forms
        of the split toric variety {$(X_{\Sigma,K}, \T_{N,K})$} if and only if for every homomorphism
        $\varphi\colon \calG\to \Aut_\Sigma{(K)}$ and cone $\tau\in\Sigma$, the points
        $\{ x_{\varphi_g({\tau})} \mid g\in \calG\}$
       indexed by cones in the $\calG$-orbit of $\tau$ lie in an affine open subset of
       $X_{\Sigma,K}$.
 \end{enumerate}
\end{theorem}

\begin{proof}
  The first statement follows, {\it mutatis mutandis}, from arguments given in
  Section~\ref{S:GalCoh}, and a small calculation involving the unit cocycle.

  The condition in the second statement is necessary.
  Indeed, in Section~\ref{S:partition}, we show that a homomorphism $\varphi\colon\calG\to\Aut_\Sigma$
  gives a cocycle $\One_\varphi\colon g\mapsto (1,\varphi_g)$, and thus a twisted form
  $X_\varphi$ of $X_{\Sigma,K}$.
  Thus the condition that the points  $\{x_{\varphi_g({\tau})} \mid g\in \calG\}$ lie in an
  affine open subset coincides with the condition in
  Proposition~\ref{P:GalDes} for $X_\varphi$ to have descent, but only for orbits of the
  distinguished points $x_\sigma$ for cones $\sigma\in\Sigma$.

  For sufficiency, suppose that we have a twisted form of $X_{\Sigma,K}$
  given by a cocycle
  $ \bc \colon g \mapsto c_g =(t_g,\varphi_g)$ with
  $\varphi\colon \calG\to\Aut_\Sigma$ the corresponding homomorphism,
  and let $x\in X_{\Sigma,K}$.
  We show there exists an affine open subset $V$ of $X_{\Sigma,K}$ containing the
  $\calG$-orbit of $x$.

  To that end, write $x=t_x  x_\sigma$ for some cone $\sigma$ of the fan $\Sigma$.
  Let $U\subset X_{\Sigma,K}$ be an affine open subset containing the
  points $\{x_{\varphi_g({\tau})} \mid g\in \calG\}$ indexed by the $\calG$-orbit of the cone $\sigma$.
  For each $g\in \calG$,
  consider the map $f_g\colon {\T_{N,K}}\to X_{\Sigma,K}$ defined by
\[
   f_g(t)\ =\ t . {^{g} x} \, ,
\]
  where $g$ acts on $x$ via the twisted action of $\calG$ on $X_{\Sigma,K}$.
  We claim that the image meets the set $U$, for every $g\in\calG$.

  If so, then $f_g^{-1}(U)$ is a non-empty open subset $U_g$ of  ${\T_{N,K}}$,
  and dense as $\T_N$ is irreducible.
  It follows that for a point $t$ in the intersection of the sets $U_g, \ g\in \calG$, one has
  $t . {^g x}\in U$ for all $g\in \calG$.
  In other words, the affine open set $t^{-1}U$ contains the orbit $\calG x$.

  To prove the claim, we show that for every $g\in\calG$
  some point of the form $t . {^g x}$ lies in $U$.
  We have
\[
    {^g x}\ =\ {^g (t_x x_\sigma)}\ =\ {^g t_x}{^g x_\sigma}\ =\ 
    {^g t_x} t_g x_{\varphi_g(\sigma)}\,.
\]
Hence for $t = ({^gt_x} t_g)^{-1} \in\T_N(K)$  
we have $t . {^g x} = x_{\varphi_g(\sigma)}$, which lies in $U$, by hypothesis.
  This completes the proof.
\end{proof}

\begin{cor}\label{C:quadratic}
 If $K/k$ is a quadratic extension, the lattice $N$ has rank $2$, or the fan
 $\Sigma$ is quasiprojective, then $H^1(K/k,\TAutS)$ classifies $k$-forms of the split
 toric variety $X_{\Sigma,K}$.
\end{cor}

\begin{proof}
 The result for $K/k$ quadratic follows from Proposition~\ref{P:GalDes} and
 W{\l}odarczyk's result~\cite{Wl93} that any pair of points in a toric variety is
 contained in an affine open subset. 
 Since any fan in a rank $2$ lattice is quasiprojective, the rest of the statement
 follows by Proposition~\ref{P:GalDes}.
\end{proof}

Huruguen~\cite{Hur} gives an example of a three-dimensional toric variety and a degree three
field extension which does not satisfy descent.
This shows that this result (Corollary~\ref{C:quadratic}, also obtained by Huruguen) is best possible.

While we have considered twisted forms of toric varieties associated to a fan in a given
finite extension of $k$, the twisted forms from the algebraic closure of $k$ are also given by
the Galois cohomology groups 
\[
   H^1(k\, ,\TAutS)\ :=\ \colim_{K \subset \overline{k}}\ H^1(\Gal(K/k) , \TAutS(K)) \, .
\]

\subsection{Real forms of \texorpdfstring{$\P^1$}{P1}}\label{S:realP1}

 Consider the projective line when $k=\R$.
 Here, $N=\Z$ and the fan $\Sigma$ has three cones:
 the positive integers $\sigma_+$, the negative integers $\sigma_-$, and
 their intersection $\{0\}$.
 Identify $M$ with $\Z$ via the pairing $\langle u,v\rangle=uv$, where $u\in M$ and
 $v\in N$.
 Then $\sigma_\pm^\vee=\sigma_\pm$ and $\{0\}^\vee=\Z$.
 Writing an element $u\in M$ multiplicatively as $z^u$, we have
\[
   \Z[\sigma_+^\vee]\ =\ \Z[z]\,,\qquad
   \Z[\sigma_-^\vee]\ =\ \Z[z^{-1}]\,,\qquad\mbox{and}\qquad
   \Z[\{0\}^\vee]\ =\ \Z[z,z^{-1}]\,,
\]
 which gives the usual construction of $\P^1$ by gluing two copies of the affine line
 $\A^1_+$ and $\A^1_-$ along the common torus $\G_m$ where $x\in\G_m$ is mapped to
 $x\in\A^1_+$ and to $x^{-1}\in\A^1_-$.

 Over $\C$, this gives the familiar Riemann sphere
\[
 \begin{picture}(70,90)
    \put(0,10){\includegraphics[height=70pt]{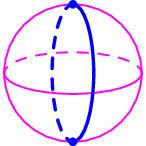}}
    \put(29.5,83){$\infty$}
    \put(32,0){$0$}
   \end{picture}
\]
 where $\G_m(\C)=\units{\C}$ is the complement of the poles $\{0,\infty\}$, which are the
 origins of $\A^1_+(\C)$ and $\A^1_-(\C)$, respectively.
 A twisted form of $\P^1_\C$ is given by an anti-holomorphic involution that normalizes the
 action of the torus $\units{\C}$.
 By either Corollary~\ref{C:quadratic} (as $k=\R$) or Theorem~\ref{Th:classification}(1) (as
 $\P^1$ is projective), twisted forms of $\P^1$ are equivalent to real forms, and both are
 in natural bijection with the Galois cohomology set $H^1(\C/\R,\TAutS)$.

 For $\P^1$, $\Aut_\Sigma=\{\pm I\}$, where $I$ is the identity map on $\Z$ and $-I$ is
 multiplication by $-1$.
 The toric automorphism group of $\P^1_\C$ is
 $\TAutS(\C):= \units{\C} \rtimes\{\pm I\}$, where $\{\pm I\}$ acts on $\units{\C}$ by
 $-I$ sending $t\in \units{\C}$ to $t^{-1}$.
 The Galois group $\calG:=\Gal(\C/\R)$ is $\{e,g\}$, where $e$ is the identity and $g$ is complex
 conjugation.

 For any cocycle $\bc$, $c_e=(1,I)$,
 so a cocycle $\bc$ is determined by $c_g=(\lambda,\varphi)\in\C^*\rtimes\{\pm I\}$.
 Suppose that $\varphi=I$.
 By the cocycle condition~\eqref{Eq:cocycle},
\[
   (1,I)\ =\ c_e\ =\ c_{g^2}\ =\ c_g\cdot {^g}c_g\ =\
   (\lambda,I)\cdot{^g}(\lambda,I)\ =\
   (\lambda\overline{\lambda},I)\,.
\]
 Thus $\lambda\overline{\lambda}=1$ and so $\lambda\in S^1$.
 Let $b^2=\lambda$.
 Then the cocycle given by
\[
  (b^{-1},I)\cdot(\lambda,I)\cdot {^g}(b,I)\ =\
  (b^{-1}\overline{b}\lambda,I)\ =\ (b^{-2}\lambda,I)\ =\ (1,I)
\]
 is cohomologous to $\bc$.
 Thus the unit cocycle $\One$ is the unique element in the Galois cohomology set
 $H^1(\C/\R,\TAutS)$ with $\varphi_g=I$.
 The corresponding twisted form is $\P^1_\C$ with the usual complex conjugation, which is
 reflection in the plane of the Greenwich meridian.
 The fixed points of this involution are the real-valued points of $\P^1_\R$ and they include
 the two poles.

 Suppose now that $\varphi=-I$.
 By the cocycle condition~\eqref{Eq:cocycle},
\[
  (1,I)\ =\  c_g\cdot {^g}c_g\ =\ (\lambda,-I)\cdot{^g}(\lambda,-I)\ =\
  (\lambda \overline{\lambda}^{-1}, I)\,,
\]
 as ${^{-I}}\lambda=\lambda^{-1}$.
 We conclude that $1=\lambda \overline{\lambda}^{-1}$ and so $\lambda\in \units{\R}$ is real.
 Let us investigate cohomologous cycles.
 For $b\in \units{\C}$,
\[
   (b, \pm I)^{-1} \cdot (\lambda,-I) \cdot {^g}(b,\pm I) =
   ( ((b\overline{b})^{-1}\lambda)^{\pm 1}, -I)\,.
\]
 Since these are all the cohomologous cycles and $b\overline{b}$ is a positive real number, we
 see that there are exactly two elements of $H^1(\C/\R,\TAutS)$ with $\varphi=-I$, namely
\[
    \bc^+\ \colon \ c_g\ =\ (1,-I)
   \qquad\mbox{and}\qquad
    \bc^-\ \colon\ \ c_g\ =\ (-1,-I)\,.
\]
 Both give the same twisted form of $\units{\C}$ in which ${^g}t=\overline{t}^{-1}$ for
 $t\in \units{\C}$.
 This is the real non-split form of $\units{\C}$ whose fixed points are $S^1$.

 We consider the corresponding twisted forms of $\P^1_\C$.
 For the cocycle $\bc^+$, the anti-holomorphic involution sends $t\mapsto\overline{t}^{-1}$
 for $t\in \units{\C}$ and it interchanges the poles.
 This is reflection in the equator and has fixed point set $S^1$.
 For the twisted form given by the cocycle $\bc^-$, the anti-holomorphic involution sends
 $\units{\C} \ni t\mapsto -\overline{t}^{-1}$ and it interchanges the poles.
 This is the antipodal map and it has no fixed points.

 Figure~\ref{F:P1_R} shows these three real forms of the toric variety $\P^1_\C$.
\begin{figure}[htb]
\[
  \begin{picture}(70,130)(0,-40)
   \put(0,10){\includegraphics[height=70pt]{figures/RP1.eps}}
   \put(29.5,83){$\infty$}
   \put(32,0){$0$}
   \put(20,-12){$t\mapsto \overline{t}$}
   \put(1,-25){Reflection in}
   \put(1,-40){Meridian $(\One)$}
  \end{picture}
   \qquad\qquad
  \begin{picture}(70,130)(0,-40)
   \put(0,10){\includegraphics[height=70pt]{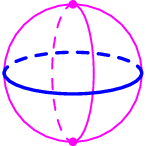}}
   \put(29.5,83){$\infty$}
   \put(32,0){$0$}
   \put(20,-12){$t\mapsto \overline{t}^{-1}$}
   \put(1,-25){Reflection in}
   \put(0,-40){Equator $(\bc^+)$}
  \end{picture}
   \qquad\qquad
  \begin{picture}(70,130)(0,-40)
   \put(0,10){\includegraphics[height=70pt]{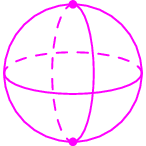}}
   \put(29.5,83){$\infty$}
   \put(32,0){$0$}
   \put(20,-12){$t\mapsto -\overline{t}^{-1}$}
   \put(-5,-25){Antipodal map}
   \put(23,-40){$(\bc^-)$}
  \end{picture}
\]
 \caption{Real forms of $P^1_\C$.}
 \label{F:P1_R}
\end{figure}
 The third real form is the Brauer-Severi variety.
 It is one of two real forms of the projective line, the other being the usual real projective line, $\R\P^1$.
 We have $\P^1_{\bc^+}\simeq \P^1_{\One}=\R\P^1$, but
 not as toric varieties.
 For example, the real points of $\P^1_{\bc^+}$ do not include the torus fixed points
 $\{0,\infty\}$, while these are real points of $\R\P^1$.

%
\subsection{A partition of Galois cohomology}\label{S:partition}

Let $\Sigma$ be a fan in a lattice $N$ and $K/k$ a Galois extension with Galois group
$\calG$.
Then we have a short exact sequence of groups
\[
  1 \rightarrow
   \T_N(K)\ \xrightarrow{\ \iota\ }\
   \TAutS(K)\ \xrightarrow{\ \pi\ }\
   \Aut_\Sigma
   \rightarrow 1
\]
which induces an exact sequence of Galois cohomology sets~\eqref{Eq:exact_sequence}
 \begin{equation}\label{Eq:LES}
  H^1(K/k,\T_N)\ \xrightarrow{\ \iota\ }\
  H^1(K/k,\TAutS)\ \xrightarrow{\ \pi\ }\
   H^1(K/k,\Aut_\Sigma)\ \xrightarrow{\ \delta\ }\
  H^2(K/k,\T_N)\ .
 \end{equation}
This begins with $H^1(K/k,\T_N)$, which is trivial by Hilbert's Theorem 90.

Since $\calG$ acts trivially on $\Aut_\Sigma$, the cohomology set $H^1(K/k,\Aut_\Sigma)$ consists of conjugacy
classes of homomorphisms $\varphi\colon\calG\to\Aut_\Sigma$ to which one associates twisted forms
${_\varphi\scrT}$ of $\T_N$, as in Proposition~\ref{P:tori}.

Given a twisted form $Y_K$ of the split toric variety $X_{\Sigma,K}$, the image $\varphi$ of
its cocycle $\bc$ under the composition $\TAutS\to\Aut_\Sigma\to\Aut(N)$ determines the
twisted torus ${_\varphi\scrT}$ acting on $Y_K$.
The fiber $\pi^{-1}(\varphi)$ above $\varphi$ consists of cocycles corresponding to the twisted forms
$Y'_K$ of $X_{\Sigma,K}$ with twisted torus ${_\varphi\scrT}$.
The following result identifies these fibers.

\begin{theorem}\label{Th:splitting}
  The map $\pi\colon H^1(K/k,\TAutS)\to H^1(K/k,\Aut_\Sigma)$ is surjective with
  the fiber over the conjugacy class of a homomorphism $\varphi\colon\calG\to\Aut_\Sigma$ equal
  to the quotient of the Galois cohomology set $H^1(K/k,{_\varphi\scrT})$
  by $H^0(K/k, \Aut_\Sigma) = (\Aut_\Sigma)^\calG = C_{\Aut_\Sigma}(\varphi(\calG))$,
  the centralizer of the image $\varphi(\calG) \subseteq \Aut_\Sigma$.
  That is, the Galois cohomology set $H^1(K/k, \TAutS)$ can be partitioned into 
  a disjoint union
\[
  H^1(K/k, \TAutS) = \coprod_\varphi H^1(K/k , {_\varphi\scrT}) / H^0(K/k , \Aut_\Sigma),
\]
where $\varphi$ varies over representatives of conjugacy classes of
homomorphisms $\varphi : \calG \to \Aut_\Sigma$.
  In particular, the association of $\varphi$ to the cocycle
  $\One_\varphi:=(\One,\varphi)$, where $\One\in H^1(K/k,{_\varphi\scrT})$ is the unit cocycle,
  is a section of the map $\pi$.
\end{theorem}

\begin{proof}
 Let $\varphi\colon\calG\to\Aut_\Sigma$ be a homomorphism.
 Setting $c_g:=(1,\varphi_g)$, where $1\in\T_N$ is the unit,
 gives a function $\bc\colon\calG\to\TAutS(K)$.
 This is a cocycle because for $g,h\in\calG$, we have
\[
   c_{gh}\ =\ (1,\varphi_{gh})\ =\
   (1,\varphi_g)(1,\varphi_{h})\ =\ c_g\cdot{^g}c_h\,,
\]
 as $\varphi_g$ is a group automorphism of $\T_N(K)$ so that ${^{\varphi_g}}1=1$.

 Write \DeCo{$\One_\varphi$} for this cocycle.
 If $\varphi$ and $\psi$ are conjugate homomorphisms (cohomologous cocycles in
 $H^1(K/k,\Aut_\Sigma)$), then $\One_\varphi$ and $\One_\psi$ are cohomologous and
 represent the same element in $H^1(K/k,\TAutS)$.
 In this way, we see that the association $\varphi\mapsto\One_\varphi$ gives a map
\[
   H^1(K/k,\Aut_\Sigma)\ \longrightarrow\
   H^1(K/k,\TAutS)
\]
 which is a section of the map $\pi$.
 Thus $\pi$ is surjective.
 
 The identification of the fiber $\pi^{-1}([\varphi]) \cong H^1(K/k, {_\varphi\scrT}) / H^0(K/k, \Aut_\Sigma)$
 follows from \cite[Cor.~I.\textsection5.5.2]{Serre}.
\end{proof}

\begin{rem}
 The Galois cohomology sets computed in Section~\ref{S:realP1} illustrate
 Theorem~\ref{Th:splitting}.
 For $\P^1$, $\Aut_\Sigma=\{\pm I\}$,
 and there are two homomorphisms $\Gal(\C/\R)\to \Aut_\Sigma$.
 We found a unique cocycle associated to the trivial homomorphism.
 This is a general fact, as $H^1(K/k,\T_N)=1$, by Hilbert's Theorem 90.
 On the other hand, there were two cocycles associated to the non-trivial homomorphism
 $\varphi$ which gives the real non-split form of $\units{\C}$,
 with real points ${_\varphi\scrT}(\R) = S^1$.
 In fact, we computed $H^1(\C/\R,{_\varphi\scrT}) = H^1(\C/\R,S^1) = \units{\R}/\Rpos$,
 while $H^0(\C/\R, \Aut_\Sigma) = \{ \pm I \}$ acts trivially on $\units{\R}/\Rpos$
 (this is just the statement that $t > 0$ if and only if $t^{-1} > 0$).
\hfill\qed
\end{rem}

\begin{rem}
 The section $H^1(K/k,\Aut_\Sigma)\to H^1(K/k,\TAutS)$ of the map $\pi$~\eqref{Eq:LES}
 is reflected in work of Voskresenski{\u\i}, who constructed
 toric varieties  corresponding to the cocycles $\One_\varphi=(1,\varphi)$,
 for smooth projective $\varphi(\calG)$-invariant fans $\Sigma$.
\hfill\qed
\end{rem}


We recall the following statement (see also \cite[I.\textsection2.6]{Serre}).

\begin{proposition}
\label{prop:partition}
Let $\varphi : \calG = \Gal(K/k) \to \Aut_\Sigma$ be a homomorphism,
let $L_\varphi$ denote the intermediate Galois extension $k \subset L_\varphi \subset K$
where $L_\varphi = K^{\ker \varphi}$,
and let $\overline \varphi : \Gal(L_\varphi / k) = \Gal(K/k) / \ker \varphi \to \Aut_\Sigma$
be the map induced by $\varphi$.
Then $H^1(K/k ,\, {_\varphi\scrT}) \cong H^1(L_\varphi / k , \, {_{\overline{\varphi}}\scrT})$.
\end{proposition}

Note that $L_\varphi$, ${_\varphi\scrT}$, and ${_{\overline{\varphi}}\scrT}$ depend
only on the conjugacy class of $\varphi$.\medskip

\noindent{\it Proof.} 
Given a closed normal subgroup $H$ of a profinite group $G$ and a $G$-module $A$,
there is an exact sequence
 \begin{equation}
  \label{eq:inf-rest}
    1 \to H^1(G/H,A^H) \xrightarrow{\,\text{inf}\,} H^1(G,A)
     \xrightarrow{\,\text{res}\,} H^1(H,A)^{G/H}
     \xrightarrow{\,\text{tg}\,} H^2(G/H,A^H)\,,
 \end{equation}
where the indicated maps are the inflation, restriction and transgression maps associated
to the normal subgroup $H$; see \cite[Prop. 1.6.6]{NSW}.

Consider the exact sequence~\eqref{eq:inf-rest} where $G=\Gal(K/k)$, 
$H=\ker{\varphi}$ and $A= {_\varphi\scrT}(K)$.
Since $H= \ker\{\varphi\colon G\to \Aut_\Sigma\}$ and $H$ acts trivially on the lattice $N$, we
have  ${_\varphi\scrT}(K) \simeq \T_N(K)$ as an $H$-module.
With this it follows from Hilbert's Theorem 90 that
 \begin{equation}
 \label{eq:restriction-1}
   H^1(H,A)^{G/H}\ \subseteq\ H^1(H,A)\ =\ H^1(\Gal(K/L_\varphi),\T_N(K))
    \ =\ H^1(K/L_\varphi,\G_m^n)\ =\ 1\, .
 \end{equation}
Furthermore,
 \begin{equation}
  \label{eq:fixed-1}
  A^H\ =\ {_\varphi\scrT}(K)^{\Gal(K/L_\varphi)}\ =\ {_{\overline{\varphi}}\scrT}(L_\varphi)
 \end{equation}
as a $G/H = \Gal(L_\varphi/k)$-module.
Then \eqref{eq:inf-rest}, \eqref{eq:restriction-1}, and \eqref{eq:fixed-1} give an
isomorphism 
\[
\begin{split}
  H^1(K/k,{_\varphi\scrT})\ =\ H^1(G,A) \ & \cong \ H^1(G/H,A^H) \\
   & = \ H^1(\Gal(L_\varphi/k),\, {_{\overline{\varphi}}\scrT}(L_\varphi))
   \ =\ H^1(L_\varphi/k,\, {_{\overline{\varphi}}\scrT})\,.
 \qed
\end{split}
\]

\begin{rem}
In \cite{D03,D04} the subgroup $\Aut_\Sigma \subset \TAutS$ is called
the group of multiplicative automorphisms of $(X_\Sigma, \T_N)$
and the subgroup $\T_N \subset \TAutS$ is called the group of elementary toric automorphisms.
Accordingly, a real structure $\bc \in H^1(\C/\R, \TAutS)$ is called a multiplicative
real structure if $\bc \in H^1(\C/\R, \Aut_\Sigma) \subset H^1(\C/\R, \TAutS)$.
Theorem~4.1.1 of \cite{D04} states that any toric real structure on a complex toric
variety $X$ in which the set of real points of $X$ is nonempty is, up to conjugation,
a multiplicative real structure.

For more general field extensions, if $\bc \in H^1(K/k, \TAutS)$ is such that the corresponding twisted form of
$(X_{\Sigma,K}, \T_{N,K})$ 
descends to a $k$-variety, then the open dense orbit $\calO_0 \subset X_{\Sigma, K}$ contains a
$k$-rational point 
if and only if $\bc \in H^1(K/k, \Aut_\Sigma)$.
If $\bc \in H^1(K/k, \Aut_\Sigma)$ then $x_0 \in \calO_0$, the distinguished point in the dense
orbit, is fixed by $\calG$. 
Conversely, if a $k$-rational point $y = t x_0 \in \calO_0$ is fixed by $\calG$ then,
writing $c_g = (t_g, \varphi_g)$ for each $g \in \calG$,
$t x_0 = y = {^{g_\varphi} y} = t_g {^g t} x_0$, meaning $t^{-1} t_g {^g t} = 1$; conjugating by $t$
takes $\bc$ to a cocycle with each $t_g=1$.

More generally, if $y$ is any $k$-rational point of $X_{\Sigma,K}$ then for each
$g \in \calG$, 
$\varphi_g$ fixes the cone $\tau$ such that $y \in \calO_\tau$.
Writing $y = t x_\tau$, the same computation shows conjugating by $t$ takes $\bc$ to a cocyle
with each $t_g$ in the stabilizer of $x_\tau$. 
\hfill\qed
\end{rem}

\subsection{Compactifications of torsors}

Suppose that $\Sigma=\{0\}$.
Then $X_\Sigma=X_{\{0\}}$ is simply $\T_N$.
Since $\Sigma$ is preserved by every homomorphism
$\varphi\colon \calG=\Gal(K/k)\to\Aut(N)$,
for every $k$-form ${_\varphi\scrT}$ of the torus $\T_{N,K}$, there are $k$-forms of
$X_{\{0\}}$, and these are in bijection with $H^1(K/k,{_\varphi\scrT}) / H^0(K/k, \Aut(N))$.
These are pairs $(Y,{_\varphi\scrT})$ of $k$-varieties with $Y\simeq{_\varphi\scrT}$, but
where the Galois action on $Y_K\simeq \T_{N,K}$ is not necessarily that on the group
$\scrT_{\varphi,K}\simeq \T_{N,K}$.
Such $k$-forms of $X_{\{0\}}$ with torus ${_\varphi\scrT}$
are \DeCo{{\sl torsors}} over ${_\varphi\scrT}$.

We restate Theorem~\ref{Th:splitting} giving an arithmetic version of the fundamental theorem
of toric varieties---that normal varieties over an algebraically closed field equipped with
the action of a dense split torus are classified by fans.

\begin{theorem}
 Suppose either that $\Sigma$ is a quasiprojective fan or that $K/k$ is a quadratic Galois
 extension.
 Then every torsor $(Y,{_\varphi\scrT})$ over a torus ${_\varphi\scrT}$ given by a
 homomorphism $\varphi\colon \calG\to\Aut_\Sigma$ has an equivariant compactification that
 is a $k$-form of the toric variety $X_\Sigma$.
\end{theorem}

This completes the classification of quasiprojective compactifications of torsors, as
every arithmetic toric variety $(Y,\scrT)$ gives a $\varphi$-invariant fan $\Sigma$ and
$(Y,\scrT)$ is the closure of the torsor $(Y_0,\scrT)$  where $Y_0$ is the dense orbit.

%
%
\section{Galois cohomology and the class group}\label{S:classgroup}

We use the presentation of the class group appearing in the
quotient construction of Section~\ref{Cox} and the fibration of Theorem~\ref{Th:splitting}  
to compute the Galois cohomology set $H^1(K/k,\TAutS)$, and then classify projective spaces
with cyclic Galois groups.

%
%
\subsection{Galois cohomology and the class group}\label{S:GCCG}

Let $\Sigma(1)$ be the set of rays in the fan $\Sigma$, and let $\symg{\Sigma(1)}$
be the group of permutations of $\Sigma(1)$.
Then ${\Aut}_\Sigma$ is naturally a subgroup of $\symg{\Sigma(1)}$.
As in Section~\ref{Cox}, we assume that the cones of $\Sigma$ span a full rank sublattice of
$N$.
We obtain a short exact sequence
 \begin{equation}
   \label{eq:ses2}
    0\ \longrightarrow\  M\  \longrightarrow\  \Z^{\Sigma(1)}\
      \longrightarrow\  \cl{\Sigma}\  \longrightarrow\  0\,.
 \end{equation}
The torsion subgroup $\cl{\Sigma}_{\text{tor}}$ of $\cl{\Sigma}$ is isomorphic to 
$\Z/a_1 \Z \times \cdots \times \Z/a_r \Z$,
for some integers $a_1\geq \cdots \geq a_r \geq 2$.
We will assume that the field $K$ satisfies
 \begin{equation}
   \label{eq:assumption}
   \Ext^1_\Z(\cl{\Sigma},\units{K})\ =\ 0\,.
 \end{equation}
That is, the equations $z^{a_i} - \lambda =0$ have solutions in $\units{K}$ for all
$i=1,\ldots,r$ and $\lambda \in \units{K}$.
This assumption~\eqref{eq:assumption} holds when $\cl{\Sigma}$ is free or when 
$K$ is an algebraic closure of $k$.
In practice, we may  assume that $K$ satisfies \eqref{eq:assumption} whenever keeping track of
the splitting field of the toric variety is not relevant,
since \eqref{eq:assumption} holds for sufficiently large extensions $K/k$.

Under this assumption,~\eqref{eq:ses2} induces the exact sequence
 \begin{equation}\label{eq:nses}
   1\ \longrightarrow\  \hom{}{\cl{\Sigma}}{\units{K}}
    \ \longrightarrow\ \hom{}{\Z^{\Sigma(1)}}{\units{K}}
    \ \longrightarrow\ \hom{}{M}{\units{K}}\ \longrightarrow\ 1\,.
 \end{equation}
Let  $\gm^{\Sigma(1)}$ be the torus $\spec(\Z[\Z^{\Sigma(1)}])$ and
$G_\Sigma$ be the abelian group scheme $\spec(\Z[\cl{\Sigma}])$.
Then we may rewrite~\eqref{eq:nses} as
 \begin{equation}\label{eq:twses}
   1\ \longrightarrow\ G_\Sigma(K)\ \xrightarrow{\ \imath_K\ } \G^{\Sigma(1)}_m({K})
    \ \xrightarrow{\ \phi_K\ }\ \T_N(K)\ \longrightarrow\ 1\,.
 \end{equation}
Since $\Aut_\Sigma \subset \symg{\Sigma(1)}$, $\Aut_\Sigma$ acts on $\G_m^{\Sigma(1)}$
and~\eqref{eq:twses} is $\Aut_\Sigma$-equivariant.
We may use any homomorphism $\varphi \colon \calG=\Gal(K/k) \to \Aut_\Sigma $
to compatibly twist the $\calG$-action on~\eqref{eq:twses}, 
obtaining a short exact sequence of twisted group schemes,
 \begin{equation}
   \label{eq:tori}
   1\ \longrightarrow\  G_{\Sigma,\varphi}(K)\ \xrightarrow{\ \imath_K\ }\
         \G^{\Sigma(1)}_{m,\,\varphi}(K)\ \xrightarrow{\ \phi_K\ }\
         {_\varphi\scrT}(K)\ \longrightarrow\ 1\,.
 \end{equation}

By Theorem~\ref{Th:splitting}, the quotient $H^1(K/k, {_\varphi\scrT}) / H^0(K/k, \Aut_\Sigma)$ is the fiber 
above the homomorphism $\varphi \in H^1(K/k, \Aut_\Sigma)$ under the projection from 
$H^1(K/k, \TAutS)$.
We will use the sequence \eqref{eq:tori} to describe $H^1(K/k, {_\varphi\scrT})$
when the extension $K/k$ is cyclic.

We begin by establishing some notation.
The orbit decomposition $\Sigma(1) = O_1 \amalg \cdots \amalg O_s$ of $\Sigma(1)$ under
the action of $\calG$ via $\varphi \colon \calG \to \Aut_\Sigma \subset \symg{\Sigma(1)}$
gives a decomposition
 \[
    \Z^{\Sigma(1)}\ =\ \Z^{O_1} \oplus \cdots \oplus \Z^{O_s}\,.
 \]
For each $i=1,\dotsc,s$, choose a representative $v_i\in O_i$ and let $\calG_i\subset\calG$ be
its stabilizer, so that $O_i=\calG/\calG_i$.
This gives a decomposition of $\G^{\Sigma(1)}_{m,\,\varphi}(K)$ as a $\Z[\calG]$-module,
 \begin{equation}
   \label{eq:decomp}
    \G^{\Sigma(1)}_{m,\,\varphi}(K) \ \cong \
    \left( \Z[\calG/\calG_1]\otimes \units{K} \right)\ \times \cdots \times\
    \left( \Z[\calG/\calG_s]\otimes \units{K} \right)\, .
 \end{equation}

Let $\{ g_{i,j} \mid j=1,\ldots,m_i\}$ be a set of representatives for
$\calG/\calG_i$ and write an element $\alpha$ in $\Z[\calG/\calG_i]\otimes \units{K}$ as
$\alpha = \sum_{j=1}^{m_i}\ [g_{i,j}]\otimes \lambda_j$, where $[g_{i,j}]$ is the coset
of $g_{i,j}$.
Consider the map
\[
    \units{K}\ \longrightarrow\ \Z[\calG/\calG_i] \otimes \units{K}\,,
    \quad\mbox{ where }\quad
    \lambda\ \longmapsto\ \sum_{j=1}^{m_i} [g_{i,j}] \otimes 
        {\rule{0pt}{11pt}^{g_{i,j}}}\!\lambda \, .
\]
We leave the reader to check that if $\lambda \in \units{(K^{\calG_i})}$
then $\sum [g_{i,j}] \otimes {\rule{0pt}{11pt}^{g_{i,j}}}\!\lambda$ is $\calG$-fixed.
Write $\Delta_i$ for the restriction 
$\units{(K^{\calG_i})} \to \{ \Z[\calG/\calG_i] \otimes \units{K} \}^{\calG}$.
One may check that $\Delta_i$ is an isomorphism\str{,}
and does not depend on the choices of the representatives $g_{i,j}$.

We state the main results of this section.

\begin{theorem} \label{thm:h1}
  Let $\cl{\Sigma}$ be as defined in $\eqref{eq:nses}$ and suppose that $K/k$ 
  is a finite cyclic extension with $\Ext^1_\Z ( \cl{\Sigma}, \units{K}) = 0$.
  Then for any homomorphism $\varphi : \calG \to \Aut_\Sigma$,
\[
    H^1(K/k, \scrT_{\varphi} )\ \cong \
      \frac{G_{\Sigma,\varphi}(k)\cap \Img{\scrN_{\G^{\Sigma(1)}_{m,\varphi}}}}{\Img{\scrN_{G_{\Sigma,\varphi} }}}\,.
\]
\end{theorem}

Here, $\scrN_{\G^{\Sigma(1)}_{m,\varphi}}$ and $\scrN_{G_{\Sigma,\varphi} }$ are the norm
homomorphisms of Section~\ref{section: norm brauer}.
We determine $\Img{\scrN_{\G^{\Sigma(1)}_{m,\varphi} }}$.

\begin{theorem} \label{thm:calt} 
  Let $\cl{\Sigma}$ and $\Delta_i,\calG_i$ for $i=1,\dotsc,s$ be as above with $K/k$ a
  finite Galois extension (not necessarily cyclic) with
  $\Ext^1_\Z ( \cl{\Sigma}, \units{K}) = 0$.
  Then for $\varphi : \calG \to \Aut_\Sigma$,
 \[
    \Img{\scrN_{\G_{m,\varphi}^{\Sigma(1)}}}\ =\ \prod_{i=1}^s \Delta_i (\Img{N_{K/K^{\calG_i}}})\,.
 \]
\end{theorem}

Before proving these results, we compute the cohomology of the middle term
in~\eqref{eq:tori}. 

\begin{lemma}
 Let $\varphi \colon \calG \to \Aut_\Sigma\subset \symg{\Sigma(1)}$ and $\calG_i, i=1,\ldots,s$
 be as above.
 Then
 \begin{equation} \label{eq:prop}
   H^r(K/k , \G^{\Sigma(1)}_{m,\,\varphi}(K)) \ \cong\  
    \prod_{i=1}^s\ H^r(\calG_i , \units{K})\,,
\end{equation}
for all $r \geq 0$. In particular,
 \begin{eqnarray}
   \nonumber
  H^0(K/k , \G^{\Sigma(1)}_{m,\,\varphi}(K))\ &  \cong& \prod_{i=1}^s\ \units{(K^{\calG_i})}\,,\\
   \label{Eq:triv}
  H^1(K/k , \G^{\Sigma(1)}_{m,\,\varphi}(K))\ & =&  1 \, , \quad \text{and} \\
    \nonumber
  H^2(K/k, \G^{\Sigma(1)}_{m,\,\varphi}(K))\ & \cong& \prod_{i=1}^s\ \brauer{K^{\calG_i} | K}\,. 
\end{eqnarray}
\end{lemma}

\begin{proof}
 It follows from Shapiro's lemma that
 $H^r(K/k, \Z[\calG/\calG_i]\otimes \units{K})\ \cong\ H^r(\calG_i, \units{K})$, 
 for all $r\geq 0$ and $i=1,\ldots, s$, since
 $\Z[\calG/\calG_i]\otimes \units{K} \cong \Ind_{\calG_i}^\calG\left(\Res^\calG_{\calG_i}(\units{K})\right)$.
 Applying this to each factor in \eqref{eq:decomp} proves \eqref{eq:prop}.

 We have $H^0(\calG_i , \units{K}) = (\units{K})^{\calG_i} = \units{(K^{\calG_i})}$ by the definition of $H^0$.
 The vanishing of $H^1$ follows from Hilbert's Theorem 90.
 Lastly, the identification of $H^2$ follows from the canonical
 isomorphism $H^2(\calG_i, \units{K}) \simeq \brauer{K^{\calG_i}|K}$,
 as explained in Section~\ref{section: norm brauer}.
\end{proof}

\begin{proof}[Proof of Theorem~$\ref{thm:h1}$]
The long exact sequence of cohomology coming from \eqref{eq:tori} includes
\[
  H^1(K/k, \G_{m,\varphi}^{\Sigma(1)}(K))\ \to\
  H^1(K/k , {_\varphi\scrT}(K))
   \ \to\
   H^2(K/k , G_{\Sigma,\varphi}(K))\ \xrightarrow{\imath_2}\
   H^2(K/k , \G_{m,\varphi}^{\Sigma(1)}(K))\,.
\]
By~\eqref{Eq:triv} we have $H^1(K/k , \G_{m,\varphi}^{\Sigma(1)}(K))=1$ and so,
\[  
   H^1( K/k , {_\varphi\scrT} )\ =\ H^1(K/k , {_\varphi\scrT}(K))\ \simeq\ \ker \imath_2\, .
\]
Since $\calG$ is cyclic, by~\eqref{Eq:H^2} we have,
\[
  H^2(K/k,G_{\Sigma,\varphi}(K))\ =\ 
   \frac{G_{\Sigma,\varphi}(k)}{\Img \scrN_{G_{\Sigma,\varphi}}}
  \qquad\mbox{\rm and}\qquad
  H^2(K/k , \G_{m,\varphi}^{\Sigma(1)}(K))\ =\ 
  \frac{\G_{m,\varphi}^{\Sigma(1)}(k)}{\Img \scrN_{\G_{m,\varphi}^{\Sigma(1)}}}\ .
\]
The result follows immediately.
\end{proof}

\begin{proof}[Proof of Theorem~$\ref{thm:calt}$]
For $i=1,\dots,s$, define $\theta_i : \Z[\calG/\calG_i] \otimes \units{K} \to \units{K}$
by
\[
  \theta_i \biggl(\sum_{j=1}^{m_i} [g_{i,j}] \otimes \lambda_j \biggr)\ =\
    \prod_{j=1}^{m_i}\, {\rule{0pt}{11pt}^{g_{i,j}^{-1}}}\! \lambda_j\,.
\]
We claim that if
$\alpha_i\in\Z[\calG/\calG_i]\otimes\units{K}(\subset\G_{m,\varphi}^{\Sigma(1)})$, then   
\[
  \scrN_{\G_{m,\varphi}^{\Sigma(1)}}(\alpha_i)\ =\ \Delta_i ( N_{K / K^{\calG_i}} (\theta_i(\alpha_i)) )
   \ \in\ \Delta_i( \Img N_{K / K^{\calG_i}}) \, .
\]
Indeed, this follows directly (albeit tediously) from the definitions of $\scrN$ and
$\theta_i$, using the expansion of $\alpha_i$,
\[
     \alpha_i\ =\ \sum_{j=1}^{m_i} [g_{i,j}] \otimes \lambda_j\;.
\]
Hence for an arbitrary element $\alpha = (\alpha_1,\dots,\alpha_s) \in
  \G_{m,\varphi}^{\Sigma(1)}$ (with $\alpha_i\in\Z[\calG/\calG_i]\otimes\units{K}$),
\[
  \scrN_{\G_{m,\varphi}^{\Sigma(1)}}(\alpha)\ =\ \prod_{i=1}^s \scrN_{\G_{m,\varphi}^{\Sigma(1)}}(\alpha_i)
   \ \in\ \prod_{i=1}^s \Delta_i( \Img N_{K / K^{\calG_i}}) \, .
\]
Conversely, suppose $\lambda_i \in \Img N_{K / K^{\calG_i}}$ for $i=1,\dots,s$.
Let $\kappa_i \in K$ such that $N_{K / K^{\calG_i}}(\kappa_i) = \lambda_i$.
For each $i=1,\dots,s$, let $\alpha_i = [g_{i1}] \otimes\, {^{g_{i1}}\! \kappa_i}$.
Then
\[
  \scrN_{\G_{m,\varphi}^{\Sigma(1)}}(\alpha_i)\ =\ \Delta_i( N_{K / K^{\calG_i}} \theta_i(\alpha_i))
  \ =\ \Delta_i( N_{K / K^{\calG_i}} (\kappa_i))
  \ =\ \Delta_i (\lambda_i)\, .
\]
Hence $\prod \Delta_i(\lambda_i) = \scrN_{\G_{m,\varphi}^{\Sigma(1)}}(\alpha)$
for $\alpha = (\alpha_1, \dotsc , \alpha_s)$.
\end{proof}

\subsection{Arithmetic projective spaces}

We apply the results of Subsection~\ref{S:GCCG} to classify the arithmetic
forms of projective space 
considered as the toric variety $(\P^n, \G_m^{n+1} / \G_m)$.
For brevity we simply write $\P^n$.

Write $[n{+}1]$ for $\{ 0 , 1, \ldots , n \}$.
Let $N \cong \Z^n$ be the lattice $\Z[n{+}1]/\Z(1,\dotsc,1)$
and $\Sigma$ be the fan in $N$ whose cones are generated 
by proper subsets of the set of images of standard basis elements in $\Z[n{+}1]$.
The symmetric group $\symg{[n{+}1]}$ acts by permuting the coordinates
and is the group of automorphisms $\Aut_\Sigma$.

Given $n$ and $d$, let $\scrP(n{+}1,d)$ be the set of partitions
$\mbm = (d \geq m_1 \geq m_2 \geq \dotsb \geq m_s \geq 1)$ of $n{+}1$
such that each part $m_i$ divides $d$.
Write $|\mbm|$ for the length $s$ of a partition.
The set $\scrP(n{+}1,d)$ is in one-to-one correspondence with the conjugacy classes of elements 
$\sigma \in \symg{[n{+}1]}$
satisfying $\sigma^d = 1$.
Write $\scrP_1(n{+}1,d) \subset \scrP(n{+}1,d)$ for those partitions with $m_s=1$ 
and $\scrP_*(n{+}1,d) = \scrP(n{+}1,d) \setminus \scrP_1(n{+}1,d)$.

\begin{theorem}\label{prop:cyclic}
Let $K/k$ be a cyclic extension of degree $d$ with Galois group 
$\calG = \langle \xi \rangle$.
The set $E(K/k, \P^n)$ of $k$-forms of\/ $\P^n$ that split over $K$ 
is in one-to-one correspondence with 
 \[
    \scrP_1(n{+}1,d) \amalg
    \coprod_{\mbm \in \scrP_*(n{+}1,d)} 
    \frac{\units{k}\cap\bigcap_{i=1}^{|\mbm|}\Img N_{K / K^{\xi^{m_i}}}}{\Img N_{K/k}}\ .
\]
\end{theorem}

\begin{proof}
We have $E(K/k, \P^n) = H^1(K/k, \TAut_\Sigma)$.
We first describe $H^1(K/k, \Aut_\Sigma)$, then the fibers of the projection map 
$\pi : H^1(K/k, \TAut_\Sigma) \to H^1(K/k, \Aut_\Sigma)$.

As before, $H^1(K/k, \Aut_\Sigma)$
is the set of conjugacy classes $[\varphi]$ of homomorphisms
$\varphi : \calG \cong \Z / d\Z \to \Aut_\Sigma \cong \symg{[n{+}1]}$.
The conjugacy class $[\varphi]$ is determined by the cycle type of $\varphi(\xi)$,
which is a permutation whose order divides $d$.
Hence $H^1(K/k, \Aut_\Sigma) = \scrP(n{+}1,d)$.\smallskip  

The dual sequence \eqref{eq:ses2} becomes
\[
   0\ \to\ M\ \to\ \Z^{[n{+}1]}\ \to\ \dual{\Z(1,\dotsc,1)}\ \to\ 0,
\]
where $M= \{ f \in \Z^{[n{+}1]} \mid \sum_{i=0}^n \ f(i) \ = 0 \}$.
In particular, $\cl{\Sigma} = \dual{\Z(1,\dotsc,1)}$ is free, so the assumption
\eqref{eq:assumption} is satisfied. 

Let $\varphi$ be a homomorphism $\calG \to \Aut_\Sigma = \symg{[n{+}1]}$
with cycle type $\mbm = (m_1,\dots,m_s)$.
By Theorem \ref{thm:h1},
\[
  H^1(K/k, {_\varphi\scrT})\ =\ 
    \frac{ G_{\Sigma,\varphi}(k) \cap \Img \scrN_{\G_{m,\varphi}^{\Sigma(1)}} }
      { \Img \scrN_{G_{\Sigma,\varphi}} } \, .
\]
In the sequence \eqref{eq:tori}, $G_{\Sigma,\varphi}(K) \cong \units{K}$
maps into $(\units{K})^{n{+}1}$ as the diagonal $\Delta_{\units{K}}$.
For each $i=1,\dotsc,s$, the subgroup $\calG_i$ is $\langle \xi^{m_i} \rangle$.
Therefore by Theorem \ref{thm:calt},
\[
  G_{\Sigma,\varphi}(k) \cap \Img \scrN_{\G_{m,\varphi}^{\Sigma(1)}}\ \cong\ 
  \Bigl\{ x \in \units{k} \mid x \in \bigcap_{i=1}^s \Img N_{K / K^{\xi^{m_i}}} \Bigr\}
\]
and $\Img \scrN_{G_{\Sigma,\varphi}} \cong \Img N_{K/k} \subset \units{k}$.

Finally, $H^0(K/k, \Aut_\Sigma)$ acts by permuting the entries of $\G_{m,\varphi}^{\Sigma(1)}$,
so it acts trivially on the diagonal $G_{\Sigma,\varphi}$.
Hence it acts trivially on $H^2(K/k, G_{\Sigma,\varphi})$ and on the kernel $\ker \imath_2$,
which is $H^1(K/k,{_\varphi\scrT})$.
Therefore the fibers of $\pi : H^1(K/k, \TAut_\Sigma) \to H^1(K/k, \Aut_\Sigma)$
are the cohomology groups $H^1(K/k, {_\varphi\scrT})$ (rather than the quotients of these by $H^0(K/k, \Aut_\Sigma)$).

This shows that
\[
  H^1(K/k,\TAut_\Sigma)\ =\ \coprod_{\mbm \in \scrP(n{+}1,d)}
    \frac{\units{k}\cap \bigcap_{i=1}^{|\mbm|} \Img N_{K / K^{\xi^{m_i}}} }{ \Img N_{K/k} } \, .
\]

If $\Sigma(1)$ has a fixed point $v = O_{i_0}$ under the action of $\calG$,
then $\calG_{i_0} = \calG$, hence $\Img N_{K/K^{\calG_{i_0}}} = \Img N_{K/k}$, and hence
$H^1(K/k , {_\varphi\scrT}) = \One_\varphi$. 
The maps $\varphi$ for which there is a fixed point are precisely those
whose conjugacy class lies in $\scrP_1(n{+}1,d)$.
This shows that $\pi : H^1(K/k, \TAut_\Sigma) \to H^1(K/k, \Aut_\Sigma) = \scrP(n{+}1,d)$
is one-to-one over $\scrP_1(n{+}1,d) \subset \scrP(n{+}1,d)$.
\end{proof}

\begin{cor}\label{C:prime}
Let $K/k$ be a cyclic extension of prime degree $d$.
Then
\[
  E(K/k,\P^n)\ =\
  \begin{cases}
    \scrP_1(n{+}1, d) \amalg \brauer{k|K}, & \text{if\ }d \mid n{+}1 \\
    \scrP(n{+}1, d), & \text{otherwise.}
  \end{cases}
\]
\end{cor}

\begin{proof}
Every element of $\scrP(n{+}1,d)$ has the form $(d,\dotsc,d,1,\dotsc,1)$.
If $d$ does not divide $n{+}1$ then $\scrP(n{+}1,d) = \scrP_1(n{+}1,d)$ and the result follows.
Otherwise $\scrP_*(n{+}1,d)$ contains the single element $\mbm_* := (d,d,\dotsc,d)$.
For $[\varphi] = \mbm_*$ we have $\calG_i = \{1\}$ for $i=1,\dots,s$.
Hence $K^{\calG_i} = K$ and $\Img N_{K / K^{\calG_i}} = \units{K}$, for $i=1,\dots,s$.
It follows that
 \begin{equation}\label{Eq:prime}
   H^1(K/k, {_\varphi\scrT})\ =\ \frac{ \units{k} }{ \Img N_{K/k} }\ \cong\ \brauer{k|K},
 \end{equation}
which completes the proof.
\end{proof}

For example,
$E(\C/\R,\P^1)=\scrP_1(2,2)\amalg\brauer{\R|\C}=\{(1,1)\}\amalg(\units{R}/\Rpos)$,
giving the three real forms of $\P^1$ computed in Section~\ref{S:realP1}.



\section{Arithmetic Toric Surfaces}
\label{sec:surf}

All toric surfaces are quasiprojective, and so by Corollary~\ref{C:quadratic} arithmetic
toric surfaces are classified by Galois cohomology.
By Theorem~\ref{Th:splitting} the computation of Galois cohomology
is reduced to computing
$H^1(K/k, {_\varphi\scrT}) / H^0(K/k, \Aut_\Sigma)$ for all homomorphisms 
$\varphi\colon\Gal(K/k)\to\Aut_\Sigma\subset \Aut(N)$,
which we may assume are injective, by Proposition~\ref{prop:partition}. 
We may herefore replace $ \Gal(K/k) $ with the conjugacy class in $ \Aut_\Sigma $ of its image 
\DeCo{$\Gamma$} under $ \varphi $.

Identifying $N$ with $\Z^2$ identifies
 $ \Aut(N) $ with $ \GL(2,\Z) $. 
We will show that every finite subgroup of $\GL(2,\Z)$ occurs
as the automorphism group of a fan $\Sigma$ of a smooth complete toric surface. 
Then we will compute $H^1( K/k , {_\varphi\scrT})$ for all conjugacy classes of 
homorphisms $ \varphi \colon \calG = \Gal(K/k) \to \Aut_\Sigma $ and lastly describe
the fiber $ H^1( K/k , {_\varphi\scrT}) / H^0(K/k,\Aut_\Sigma) $.

\begin{rem} \label{rem:action}
 With these identifications,  the subgroup $ \Gamma \subset \Aut_\Sigma $  
 induces  a $\Z[\calG]$-module  structure on $\Z^2$.
 The corresponding $\calG$-module structure on 
 ${_\varphi\scrT}(K) = \Z^2\otimes\units{K}=\units{K}\times\units{K}$ is the simultaneous
 action of $ \Gamma $ on $N=\Z^2$ and its action on $ \units{K} $
 as the Galois group $ \Gal(K/k) $. 
Note the the action on $ \Z^2$ is simply the restriction of the action of $ \GL(2,\Z) $.
More precisely, if $g\in\calG$ with $\varphi(g)=(\begin{smallmatrix}a&b\\c&d\end{smallmatrix})$, and
 $(x,y)\in\units{K}\times\units{K}$, then 
 \[
    (x,y)\ \longmapsto\ {^g (x,y)}\ =\ \left(\ g(x)^a g(y)^b,\ g(x)^c g(y)^d \ \right)\,.
 \]
Similarly,  any map $\Gal(K/k)\to\GL(n,\Z)$ induces a corresponding action on $(\units{K})^n$.
\hfill\qed
\end{rem}


\subsection{Finite subgroups of \texorpdfstring{$\GL(2,\Z)$}{GL(2,Z)}}

Write $D_{2m}$ for the dihedral group of order $2m$,
 \[
    D_{2m} := 
    \langle \rho,\mbr\mid\mbr\rho=\rho^{-1}\mbr\ \text{ and }\ \rho^m=\mbr^2=e\rangle\,,
 \]
and write $C_m$ for the cyclic group of order $m$.
A \emph{maximal} finite subgroup of $\GL(2,\Z)$ is isomorphic to either $D_8$ or $D_{12}$.
Table~\ref{table:subgroups} contains a complete set of representatives for the conjugacy classes
of subgroups of $\GL(2,\Z)$ as well as their generators.
(See \cite[\textsection IX.14]{MR0340283}.)

\begin{table}[!htb]

\caption{Finite subgroups of $\GL(2,\Z)$ and their generators }
\label{table:subgroups}
\begin{tabular}{l||l}
\rowcolor[gray]{.8}
Cyclic & Dihedral \\ \hline\hline
\rowcolor[rgb]{0.96,0.96,0.9}
$C_6 = \langle A \rangle $    &  $D_{12} = \langle A, J \rangle$               \\
\rowcolor[gray]{.8}
$C_4 = \langle B \rangle  $   &  $D_{8\phantom{0}} = \langle B, J \rangle$     \\
\rowcolor[rgb]{0.96,0.96,0.9}
$C_3 = \langle A^2 \rangle $  &  $D_{6\phantom{0}} =  \langle A^2, JA \rangle $  \\
\rowcolor[gray]{.8}
$C_{2} = \langle A^3 \rangle = \langle B^2 \rangle =\langle -I \rangle $ &  $D'_{6\phantom{0}} =\langle A^2, J \rangle $ \\
\rowcolor[rgb]{0.96,0.96,0.9}
$D_{2} = \langle C \rangle  $      &    $D_{4} = \langle -I, C \rangle  $ \\
\rowcolor[gray]{.8}
$D_{2}' = \langle J \rangle  $      & $D_{4}' = \langle -I, J \rangle$     \\
\rowcolor[rgb]{0.96,0.96,0.9}
$C_1 = \langle I \rangle  $      &    \\ \hline
\end{tabular}
\qquad
\begin{tabular}{l} 
  $A$ $ =\left(\begin{array}{rr} 0 & -1 \\ 1 & 1 \end{array}\right)$
   \raisebox{-14pt}{\rule{0pt}{30pt}} \\
  $B$  $ = \left(\begin{array}{rr} 0 & -1 \\ 1 & 0 \end{array}\right)$ 
   \raisebox{-14pt}{\rule{0pt}{33pt}} \\
   $C$   $= \left(\begin{array}{rr} 1 & 0 \\ 0 & -1 \end{array}\right) $ 
   \raisebox{-14pt}{\rule{0pt}{33pt}} \\ 
  $J$   $= \left(\begin{array}{rr} 0& 1 \\ 1 & 0 \end{array}\right) $ 
   \raisebox{-9pt}{\rule{0pt}{28pt}}
\end{tabular}
\end{table}

\subsection{Smooth complete toric surfaces}

The toric surface $X_\Sigma$ corresponding to a fan $\Sigma$ in $\Z^2$ is 
complete if and only if $\Z^2$ is the union of the cones in $\Sigma$.
The surface is smooth if and only if every two-dimensional cone $\sigma$ of $\Sigma$
is generated by the primitive vectors lying in its rays.
That is, if the cone is isomorphic to the positive quadrant in $\Z^2$.

\begin{proposition}\label{prop:surfaces}
  For each conjugacy class of finite subgroups of $\GL(2,\Z)$ there is a smooth complete toric
  surface whose fan has automorphism group in that class.
\end{proposition}

\begin{proof}
 For each group in Table~\ref{table:subgroups} we display a smooth fan with that
 automorphism group.
 Figure~\ref{Fig:D12} shows the primitive generators of the one-dimensional cones in
\begin{figure}[htb]
 \begin{tabular}{|c|c|c|c|c|}\hline
  $D_{12}$&$D_{6}$&$D_{6}'$&$C_6$&$C_3$\raisebox{-1.5pt}{\rule{0pt}{14pt}}\\\hline
  \includegraphics{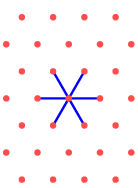} &
  \includegraphics{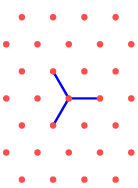} &
  \includegraphics{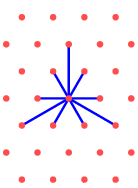} &
  \includegraphics{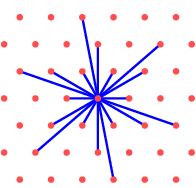} &
  \includegraphics{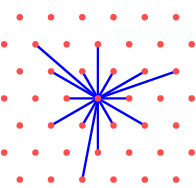}\\\hline
 \end{tabular}
 \caption{Fans with an automorphism of order three.}
 \label{Fig:D12}
\end{figure}
 complete fans with an automorphism of order three, where the lattice is drawn 
 with $D_{12}$-symmetry.
 Figure~\ref{Fig:D8} shows those whose automorphism group is a subgroup of $D_8$.
\begin{figure}[htb]
 \begin{tabular}{|c|c|c|c|}\hline
  $D_{8}$&$D_4$&$D_4'$\raisebox{-1.5pt}{\rule{0pt}{14pt}}\\\hline
  \includegraphics{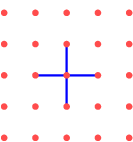}&
  \includegraphics{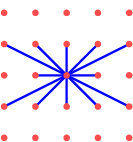}&
  \includegraphics{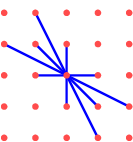}\\\hline
 \end{tabular}

 \bigskip

 \begin{tabular}{|c|c|c|c|}\hline
  $C_4$&$C_2$&$D_2$&$D_2'$\raisebox{-1.5pt}{\rule{0pt}{14pt}}\\\hline
  \includegraphics{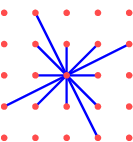}&
  \includegraphics{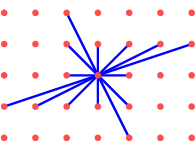}&
  \includegraphics{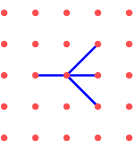}&
  \includegraphics{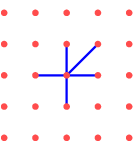}\\\hline
 \end{tabular}

 \caption{Fans whose automorphism group is a subgroup of $D_8$}
 \label{Fig:D8}
\end{figure}
 For these, we have drawn the lattice with $D_8$-symmetry.

 Each of these fans visibly exhibit the claimed symmetry groups.
 To see that they have no more automorphisms, first note that for every primitive vector
 $v$ in these fans there is an integer $a_v$ such that
\[
    a_v v\ =\ u+w\,,
\]
 where $u,v,w$ are consecutive primitive vectors in the fan.
 See \cite[\textsection 2.5]{Fulton}.
 Arranging these integers in cyclic order according to the order of their primitive
 vectors around the origin gives a cyclic sequence.
 For example, the fan for $C_3$ yields the cyclic integer sequence
\[
  (1,2,3,1,4,1,2,3,1,4,1,2,3,1,4)\,.
\]
 We leave it an exercise that the symmetry group of the  cyclic sequence of a fan
 $\Sigma$ equals the automorphism group of $\Sigma$, and that the sequences for
 the fans in Figures~\ref{Fig:D12} and~\ref{Fig:D8} have the claimed symmetry groups.
\end{proof}
 Dolgachev and Iskovskikh used arithmetic toric surfaces with the fans shown for
 $D_{12}$ and $D_8$ to study the plane Cremona group~\cite{DoIs}. 

\subsection{Calculation of Galois cohomology groups}

We compute the Galois cohomology groups 
$H^1(K/k , {}_\varphi \scrT)$.
First we consider the case that $\calG = \Gal(K/k) \subset \GL(2,\Z)$ contains an element of
order three. 
Then we consider the remaining Galois groups, all of which are subgroups of $D_8$.

\subsubsection{Galois groups with an element of order three.}

Start with the dihedral group
$ 
D_{12} = \langle \rho,\mbr\mid\mbr\rho=\rho^{-1}\mbr\ \text{ and }\ \rho^6=\mbr^2=e\rangle\,,
$
and consider the following subgroups
\begin{equation}
\label{eq:subgps}
N_0 = \langle \rho^2, \mbr \rangle \vartriangleleft D_{12} \quad \text{and} \quad
H_0=  \langle \rho^3, \mbr \rangle < D_{12} .
\end{equation}

The composition
\[ 
D_{12} \longrightarrow D_{12}/N_0 = \{ \bar 1, \bar \rho \} \cong \Z/2\Z \xrightarrow{\ \ \xi\ \ } \Aut(\Z) = \Z^{\times} ,
\]
where $ \xi $ is the alternating representation, gives $ \Z $ the structure of a $ D_{12} $-module that we denote by $ \frA $. Since $[D_{12}:~H_0]=3$, the induced module
\begin{equation}\label{eq:induced}
  \frC \ :=\ \Ind_{H_0}^{D_{12}}( \Res^{D_{12}}_{H_0} \frA)
   \ =\ \Z[D_{12}]\otimes_{\Z[H_0]} \Res^{D_{12}}_{H_0}(\frA)
\end{equation}
is a lattice of rank $3$ generated by $e_1 = 1\otimes 1$, $e_2= \rho \otimes 1$, 
and $e_3 = \rho^2\otimes 1$,
and its $\Z[D_{12}]$-module structure is determined by
 \begin{equation}  \label{eq:action}  
  \begin{array}{rclcrclcrcl}
   \rho \cdot e_1&=&e_2\,,&& \rho \cdot e_2&=&\phantom{-}e_3\,,&& \rho \cdot e_3&=&- e_1\, ,\\
      \mbr \cdot e_1&=&e_1\,,&&  \mbr \cdot e_2&=&- e_3\,,         && \mbr \cdot e_3&=&- e_2\,.
  \end{array}
 \end{equation}
The module $\frC$  comes with a natural $\Z[D_{12}]$-module epimorphism
$\pi \colon \frC \to \frA$ defined by
 \begin{equation}
  \label{eq:epimorphism}
   \pi( me_1 + ne_2 + pe_3 )\ =\ 1\cdot m + \rho \cdot n + \rho^2 \cdot p\ =\ m-n+p.
 \end{equation}
It follows that $\frB := \ker{(\pi)}$ is the sublattice $\frB =\Z\{v_1,v_2\}\subset\frC$ generated by $v_1 := e_1-e_3$ and $v_2:= e_1 + e_2$.
Using~\eqref{eq:action}, we see that the actions of $\rho$ and $\mbr$ on $\frB$
with respect to the basis $\{ v_1 , v_2 \}$ are represented by the matrices $A$ and $J$ 
of Table~\ref{table:subgroups}, respectively. 
 We conclude that the $\Z[D_{12}]$-module structure on 
$\Z^2$ given by the monomorphism $ \varphi \colon D_{12} \to GL(2,\Z) $ that sends 
$ \rho \mapsto A$, $\mbr \mapsto J $ is $\frB$ 
and that it fits into 
a short exact sequence of $\Z[D_{12}]$-modules
\begin{equation}
\label{eq:sesD}
     0\ \longrightarrow\ \frB\ \xrightarrow{\ \  j\ \ }\  
    \frC\  
  \xrightarrow{\ \ \pi\ \ }\ \frA\ \longrightarrow\ 0\, .
\end{equation}

Let $K/k$ be a Galois extension whose Galois group $ \calG = \Gal(K/K) $ has an element of
order three and comes with an embedding 
$ \varphi \colon \calG \longrightarrow  \Aut_\Sigma \subset GL(2,\Z) $,
where $ \Sigma $ is the fan of a smooth complete toric surface. 
Up to conjugacy, $ \varphi $ identifies $ \calG $ with one of the following subgroups  of 
 $D_{12} \colon D_{12}, D_6, C_6, D'_6$ or $ C_3 $.

In this context, we obtain subextensions of $ K/k $
\begin{equation}
\label{eq:subextensions}
E:= K^{\calG \cap H_0}, \quad F:= K^{\calG \cap N_0} \quad \text{and} \quad L:= K^{\calG \cap H_0 \cap N_0}
\end{equation}
that fit into a diagram
 \begin{equation}
  \label{eq:subext}
   \raisebox{60pt}{\xymatrix{
    & K \ar@{-}[d] & \\
    & L  \ar@{-}[dl] \ar@{-}[dr] & \\
    E\ar@{-}[dr] & & F \ar@{-}[dl] \\
    & k &
   }}
 \end{equation}
Consider the homomorphism
\begin{equation}
\label{eq:gamma}
   \gamma\ \colon\ \brauer{k|E}\ \longrightarrow\ \brauer{F|L}\,,
\end{equation}
obtained by base extension from $k$ to $F$, and consider the homomorphism
 \begin{equation}
   \label{eq:eta}
    \eta\ \colon\ \brauer{E|L}\ \longrightarrow\ \brauer{k|F}\,,
 \end{equation}
induced by the norm map $N_{E/k} : \units{E} \to \units{k}$. 
Denote the kernel of $\eta$ by $\etakernel$.

\begin{theorem}
\label{thm:D12}
Let $ \calG = \Gal(K/k) $ have an element of order three and 
$ \varphi \colon  \calG \to \Aut_\Sigma \subset GL(2,\Z) $ be as above. 
Then there is a canonical isomorphism
\[
   H^1(\calG, {_\varphi\scrT}(K) ) \ \cong \  \coker{(\pi^0)} \oplus \ker{(\pi^1)},
\]
where $ \coker{(\pi^0)} $ and $ \ker{(\pi^1)} $ 
are as described below.
\begin{center}
\begin{tabular}{l||c||c}
\cline{1-3}
%
\rowcolor[gray]{.8} \Blue{ $ \calG $ } &
 \Blue{$  \coker{(\pi^0)} $ } & \Blue{ $ \ker{(\pi^1)}$  }  \\ \hline\hline
\rowcolor[rgb]{0.96,0.96,0.9}
$D_{12} = \langle A, J \rangle     $ &
$ \displaystyle \frac{\brauer{F|L}}{\gamma(\brauer{k|E}) } $ & 
$ \etakernel $ \raisebox{-12.5pt}{\rule{0pt}{30.5pt}} \\
\rowcolor[gray]{.8}
$D_6 = \langle A^2 , JA \rangle $  &
$ \displaystyle \frac{\brauer{F|K}}{\gamma(\brauer{k|E}) } $ &
$ 1 $     \raisebox{-13pt}{\rule{0pt}{31pt}}  \\
\rowcolor[rgb]{0.96,0.96,0.9} 
$C_6 = \langle A  \rangle $  &
$ \displaystyle \frac{\brauer{F|L}}{\gamma(\brauer{k|E}) } $ & 
$ \etakernel $  \raisebox{-12.5pt}{\rule{0pt}{30.5pt}} \\
\rowcolor[gray]{.8}
$D'_6 = \langle A^2 , J \rangle $  &
$  \brauer{k|E}  $ &
$  1  $           \\
\rowcolor[rgb]{0.96,0.96,0.9}
$C_{3} = \langle A^2  \rangle$ &
$ \brauer{k|K}    $ &
$ 1  $ \\ \hline
\end{tabular}
\end{center}

\end{theorem}

\begin{rem}
Whenever no confusion is likely to arise, we suppress the notation $\Res^G_H(M)$ and 
write $M$ to denote the restricted module as well.  
Also, to stress the distinction between the Galois group $ \calG $ and its image under
$\varphi$, we often use $ \rho $ and $ \mbr $ instead of $ A $ and $ J $.   
\hfill\qed
\end{rem}


\begin{proof}

\noindent{\bf Step 1: General strategy.} \ \ 
Suppose that $\calG \subseteq  D_{12}$ is one of the subgroups in Table \ref{table:subgps}.
 Then 
 \[
 \scrT_{\varphi}(K) = \Res^{D_{12}}_\calG(\frB)  \otimes \units{K}
 \]
 as a
 $\Z[{\calG}]$-module.
From the long exact sequence in Galois cohomology of~\eqref{eq:sesD},
  \begin{multline}
   \label{eq:long}
      \quad H^0(\calG,\frC\otimes \units{K}) \ \xrightarrow{\ \pi^0 \ }\ 
      H^0(\calG,\frA\otimes \units{K})\ \longrightarrow\ 
     H^1(\calG,\frB\otimes \units{K}) \xrightarrow{j^1} \\
   H^1(\calG,\frC\otimes \units{K} )\ \xrightarrow{\ \pi^1\ }\ 
    H^1(\calG,\frA\otimes \units{K} ) \ \longrightarrow\ \dotsb  \quad 
\end{multline}
we extract a short exact sequence
\begin{equation}
\label{eq:ses-1}
1 \to \coker{(\pi^0)} \longrightarrow H^1(\calG, \frB\otimes \units{K} ) \xrightarrow{\ \ j^1 \ \ } \ker{(\pi^1)} \to 1 .
\end{equation}
Next we identify $ \coker{(\pi^0) } $ and $ \ker{(\pi^1)} $ and show that the sequence splits.

First, observe that
$
\Res^{D_{12}}_\calG(\frA) 
$
is given by the composition
\[
\calG \longrightarrow \calG/\calG\cap N_0 \hookrightarrow D_{12}/N_0 \xrightarrow{ \ \ \xi \ \ } \units{\Z}.
\]
Denote by $ \widehat {\frA} $ the composition $ \calG/\calG\cap N_0 \longrightarrow \units{\Z} $. 
 Since $\calG \cap N_0$ acts trivially on $\Res^\calG_{\calG\cap N_0}(\frA)$, 
 we have 
 \[ 
 H^1(\calG \cap N_0, \frA \otimes \units{K}) = H^1(\calG\cap N_0, \units{K}) = 1,
 \]
 by Hilbert's Theorem 90. As a consequence, the inflation map in the exact sequence
 \eqref{eq:inf-rest} for the normal subgroup 
 $\calG \cap N_0\lhd \calG$ gives the isomorphisms
 \begin{multline}\label{Eq:first}
  \quad  H^i\left( \calG /\calG\cap N_0, \widehat\frA \otimes \units{F} \right)\ =\ 
   H^i\left( \calG /\calG\cap N_0, (\widehat\frA \otimes \units{K})^{N_0} \right)\\
  \ \xrightarrow{\ \sim\ } \
   H^i( \calG , \Res^{D_{12}}_\calG(\frA) \otimes \units{K} )\,, \qquad
 \end{multline}
for all $ i $.

Now, observe that $ \calG \backslash D_{12}/ H_0 = \{ I \} $ holds for any $ \calG  $ appearing in the theorem, and hence the double-coset formula (see, for example, \cite[Thm.~4.2.6]{Evens}) gives
\begin{align}
\label{eq:double_coset} 
\Res^{D_{12}}_{\calG}(\frC)  & =
\Res^{D_{12}}_{\calG}\! \left( \Ind_{H_0}^{D_{12}}( \Res^{D_{12}}_{H_0}\frA) \right) \ = \
\Ind^{\calG}_{\calG \cap H_0 }\! \left( \Res_{\calG \cap H_0}^{H_0}( \Res^{D_{12}}_{H_0}\frA) \right) \\
& = \Ind^{\calG}_{\calG \cap H_0 }\! \left( \Res_{\calG \cap H_0}^{D_{12}}\frA \right) . \notag
\end{align}
Therefore,
\begin{align}
\label{eq:ResC}
H^*(\calG, \Res^{D_{12}}_\calG(\frC) \otimes \units{K}) & = 
H^*(\calG, \Ind^{\calG}_{\calG\cap H_0}(\Res^{D_{12}}_{\calG\cap H_0}(\frA)) \otimes \units{K})  \\ 
\notag
&= 
H^*(\calG \cap H_0, \Res^{D_{12}}_{\calG\cap H_0}(\frA) \otimes \units{K}),
\end{align}
for all $ * $, by Shapiro's Lemma. Next, we use \eqref{Eq:first} and \eqref{eq:ResC} to obtain the cohomology groups displayed in Table~\ref{table:subgps}.

\begin{table}[!ht]
\caption{Auxiliary data.}
\label{table:subgps}
\begin{tabular}{l||c|c|c|c|c}
%
\rowcolor[gray]{.8} 
 \Blue{$ \calG $ } &
 \Blue{$ D_{12}$ } & \Blue{ $ D_6 $  }  & \Blue{$ D'_6 $} &
 \Blue{$ C_6 $}  & \Blue{$ C_3 $}  \\ \hline\hline
\rowcolor[rgb]{0.96,0.96,0.9}
$ \calG \cap H_0 $ & $ H_0 $& $\langle JA^3 \rangle$ & $ \langle  J \rangle$  & $\langle A^3 \rangle $ & $\langle I\rangle $  
 \\
\rowcolor[gray]{.8}
$ \Res^{D_{12}}_{\calG \cap H_0} \frA $ & $ \Res^{D_{12}}_{ H_0} \frA $ & $ \xi $ & $ \Z $ & $ \xi $ & $ \Z $  
 \\
\rowcolor[rgb]{0.96,0.96,0.9} 
$ H^0( \calG, \frC\otimes \units{K} ) $ &  $ \ker(N_{L/E}) $ &  $ \ker(N_{K/E}) $& $ \units{E} $ &  $ \ker(N_{K/E}) $&  $ \units{K} $
\\
\rowcolor[gray]{.8}
$ H^1( \calG, \frC\otimes \units{K} ) $ & $ \brauer{E|L} $ & $ \brauer{E|K} $ & $ 0 $  &$ \brauer{E|K} $  & $ 0 $
\\ \hline\hline
\rowcolor[rgb]{0.96,0.96,0.9} 
$ \calG / \calG \cap N_0 $ & $ D_{12}/N_0 $ & $ \{ \overline{I}, \overline{JA} \} $ & $ \{ \overline{I} \} $& $ \{ \overline{I}, \overline{A} \} $ & \rule{0pt}{11.5pt}
$ C_3 $  \\ 

\rowcolor[gray]{.8}
$ \widehat{\frA} $ & $ \xi $  & $ \xi $ & $ \Z $ &  $ \xi$ & $ \Z $ 
 \\
\rowcolor[rgb]{0.96,0.96,0.9} 
$ H^0( \calG, \frA\otimes \units{K} ) $ & $ \ker(N_{F/k}) $ &  $ \ker(N_{F/k}) $& $ \units{k} $ &  $ \ker(N_{F/k}) $&  $ \units{k} $
\\
\rowcolor[gray]{.8}
$ H^1( \calG, \frA\otimes \units{K} ) $ & $ \brauer{k|F} $ & $ \brauer{k|F} $ & $ 0$  &$ \brauer{k|F} $  & $ 0 $
 \\
 \hline
\end{tabular}
\end{table}


For all $ \calG $ in the table, except for $ \calG = D_{12} $, the calculations follow directly from \eqref{Eq:first}, \eqref{eq:ResC} and \eqref{Eqn:periodicity}, together with
Section~\ref{section: norm brauer}.
In all cases, the arguments follow essentially the steps below.

\noindent{\bf Step 2: The case $ \calG = D_{12} $.  }\ \ 
In this case, the $\Z[D_{12}/N_0]$-module $\widehat\frA$ is isomorphic to the alternating
 representation $\xi$ of $D_{12}/N_0\cong \Z/2 $ and since $ D_{12}/N_0  = \Gal(F/k)$ we obtain
 \begin{equation}
 \label{eq:br1}
    H^1(D_{12}, \frA\otimes \units{K} )\ \cong \ H^1(\Gal(F/k), \xi \otimes \units{F} ) 
   \ \cong \ H^2(\Gal(F/k), \units{F} ) \  = \ \brauer{k|F}.
 \end{equation}
 The isomorphisms come  from \eqref{Eq:first} and  \eqref{Eqn:periodicity}, respectively, 
 and the last equality is from Section~\ref{section: norm brauer}.

Similarly,  it follows from \eqref{eq:ResC} that
 $
    H^1(D_{12}, \frC \otimes \units{K} )\ \cong \ H^1(H_0 , \Res^{D_{12}}_{H_0}(\frA) \otimes \units{K} ) .
$
 To compute this last group, we use the same arguments as for~\eqref{eq:br1}, but for 
 the normal subgroup $N_0 \cap H_0 \lhd H_0$ whose quotient is isomorphic to $\Z/2\Z$.
 This gives 
 \begin{multline}
 \label{eq:D12case} \qquad
    H^1(D_{12}, \frC\otimes \units{K})\ \cong\ 
    H^1 \left( H_0 ,   \frA \otimes \units{K}  \right)\\ 
    \cong\     H^1 \left( H_0/H_0 \cap N_0,   \widehat{\frA} \otimes \units{L}  \right)\ 
    \cong\  \brauer{E | L}\,.\qquad
 \end{multline}

 A direct calculation now shows that the homomorphism
\[
  \pi^1\ \colon\ \brauer{E|L}\ =\ H^1( D_{12}, \frC\otimes \units{K} ) 
    \ \longrightarrow\     H^1(D_{12}, \frA\otimes \units{K})\ =\ \brauer{k | F}
\]
 from~\eqref{eq:long} is the homomorphism $\eta$ \eqref{eq:eta}
 given by the norm map.   Therefore, in the case $ \calG = D_{12} $ one finds that $ \ker{\pi^1} = \etakernel = \ker{\eta}$.

Now we identify $\coker(\pi^0)$.
As in \eqref{eq:br1}, we obtain
 \[ H^0(D_{12}, \frA\otimes \units{K} ) \cong  H^0(\Gal(F/k), \xi \otimes \units{F} ) 
   \ \cong \   (\xi \otimes \units{F} )^{\Gal(F/k)}.  \]
The latter is identified with those $ a \in \units{F}$ such that $ a = \displaystyle \frac{1}{\rho(a)} $, i.e. $ N_{F/k}(a) = a \rho(a) = 1 $. Therefore, $ H^0(D_{12}, \frA\otimes \units{K} ) =  \ker{(N_{F/k})}.$
 
 Similarly, as in \eqref{eq:D12case} we obtain
\[
   H^0(D_{12}, \frC\otimes \units{K}) \cong H^0(H_0, \frA \otimes \units{K}) 
   \ = \ ( \frA\otimes \units{K})^{H_0}\,.
\]
 An element $ z \in \frA\otimes \units{K} \equiv \units{K} $ is fixed by $ H_0 $ if and only if $ z=  \mbr(z) $ and $ z = \frac{1}{\rho^3(z)} $.
 In other words, $ z $ lies in $ \units{L} $ and $ N_{L/E}(z) = z \rho^3(z) = 1 $. 
Thus we have the identification
$ H^0(D_{12}, \frC\otimes \units{K}) \cong \ker{(N_{L/E})}$.

By~\eqref{eq:epimorphism}, the map 
$\pi^0 \colon (\frC\otimes \units{K})^{D_{12}} \to (\frA\otimes \units{K})^{D_{12}}$
sends $x$ to $\pi^0(x) = \frac{x \rho^2(x)}{\rho(x)}$ and, since the identity
$x \rho^3(x) =1$ implies $\frac{1}{\rho(x)} = \rho^4(x)$, we have 
$\pi^0(x) = x \rho^2(x) \rho^4(x)=N_{L/F}(x)$.

By Hilbert's Theorem 90, the map $a\mapsto \frac{a}{\rho(a)}$ 
from $\units{F}$ to $\ker (N_{F/k})=(\frA\otimes\units{K})^{D_{12}}$ is surjective.
This induces a surjection $q \colon \units{F}\to \coker(\pi)$.
We examine the kernel of this map.

First, if $a \in \units{k} \subset \units{F}$ then $\rho(a) = a$ so $\frac{a}{\rho(a)} = 1$, hence $\units{k} \subset \ker q$.
Second, we claim $N_{L/F}(\units{L}) \subset \ker q$.
If $b \in \units{L}$ and $a = N_{L/F}(b) = b \rho^2(b) \rho^4(b)$ then
\[
  \frac{a}{\rho(a)} = \frac{b \rho^2(b) \rho^4(b)}{\rho(b) \rho^3(b) \rho^5(b)}
  = \frac{b}{\rho^3(b)} \, \rho^2\left( \frac{b}{\rho^3(b)} \right) \, \rho^4\left( \frac{b}{\rho^3(b)} \right)
  = N_{L/F}\left( \frac{b}{\rho^3(b)} \right) \, .
\]
Now $\frac{b}{\rho^3(b)} \in \ker N_{L/E}$, so $\frac{a}{\rho(a)} \in N_{L/F}(\ker N_{L/E}) = \Img(\pi^0)$ and $a \in \ker q$ as claimed.
In particular, $q$ factors through $\units{F} / N_{L/F}(\units{L}) = \brauer{F | L}$
and indeed through the quotient of this by the subgroup $\Img(\gamma)$,
where $\gamma \colon \brauer{k|E} \to \brauer{F|L}$ is induced by the inclusion
$\units{k}\hookrightarrow\units{F}$.

Conversely, suppose $a \in \units{F}$ has $q(a) = 0$, so $\frac{a}{\rho(a)} \in \Img(\pi^0) = N_{L/F}(\ker N_{L/E})$.
Let $\frac{a}{\rho(a)} = N_{L/F}(b)$, $b \in \ker N_{L/E}$.
By Hilbert's Theorem 90 there is an element $c \in \units{L}$ such that $b = \frac{c}{\rho^3(c)}$.
Then
\[
  \frac{a}{\rho(a)} = N_{L/F}\left( \frac{c}{\rho^3(c)} \right)
  = \frac{c}{\rho^3(c)} \, \rho^2\left( \frac{c}{\rho^3(c)} \right) \, \rho^4\left( \frac{c}{\rho^3(c)} \right)
  = \frac{c \rho^2(c) \rho^4(c)}{\rho(c) \rho^3(c) \rho^5(c)}
  = \frac{ N_{L/F}(c) }{ \rho N_{L/F}(c) } .
\]
But then
\[
   a  N_{L/F}(c^{-1})\ =\ \rho\left(a  N_{L/F}(c^{-1})\right)\,,
\]
 so that $a  N_{L/F}(c^{-1})\in F^{\langle\rho\rangle}=k$.
 Hence $[a] = [a N_{L/F}(c^{-1})] \ \in \ \Img(\gamma)$,
 where $\gamma \colon \brauer{k|E} \to \brauer{F|L}$ is induced by the inclusion
 $\units{k}\hookrightarrow\units{F}$.
Thus we have shown that
\begin{equation}
\label{eq:cokerpi0}
\mbox{coker}(\pi^0) \cong \frac{\brauer{F|L}}{\gamma( \brauer{k|E} )}\, .
\end{equation}
 
To complete the proof  we need to exhibit a map 
 \[ 
    s \colon \ker{(\eta)} = \ker{(\pi^1)} \ \longrightarrow \ 
     H^1(D_{12}, \frB \otimes \units{K})
 \]
 splitting the short exact sequence~\eqref{eq:ses-1}.
  Denote by $ [\alpha]_{E|L} \in \brauer{E|L} $ the class of an element $ \alpha \in \units{E} $ and observe that  $ [\alpha]_{E|L} $ is in the kernel of $ \eta $ if and only if 
\begin{equation}
\label{eq:kernel}
\alpha\rho(\alpha)\rho^2(\alpha) = N_{E/k}(\alpha) = N_{F/k}(b) = b \rho(b) 
\end{equation}
for some $ b \in \units{F} .$ Hence, $ \rho^2(b) = \mbr(b) = b$ and the following relations hold:
\begin{equation}
\label{eq:alpha}
\alpha = \frac{b}{\rho(\alpha)} \rho\left( \frac{b}{\rho(\alpha)} \right) \quad \text{and}\quad \mbr\left(\frac{\rho(b)}{\alpha} \right) = \frac{\rho(b)}{\alpha}.
\end{equation}

Now, define 
\begin{equation}
C_\rho(\alpha) = \left( \frac{\rho(\alpha)}{b}, 1 \right) \quad \text{and} \quad C_\mbr(\alpha) = \left( \frac{\rho(b)}{\alpha}, \frac{\alpha}{\rho(b)} \right).
\end{equation}
It follows directly from the relations \eqref{eq:alpha} that $ C_\rho(\alpha) $ and $ C_\mbr (\alpha) $ define a cocycle $ C(\alpha) $ for a cohomology class $ [C(\alpha) ] $ in $ H^1 ( D_{12}, \frB \otimes \units{K}) $.
Moreover, a calculation shows that the cohomology class of the cocycle $ C(\alpha) $ does
not depend on the choice of $ b \in \units{F} $  satisfying~\eqref{eq:kernel}.

Suppose that $ [\alpha]_{E|L} = [1]_{E|L} $, in other words, $\alpha = N_{L/E}(w) =w\rho^3(w)$
for some $ w \in \units{L} $.    
Here we can choose $ b = w \rho^2(w) $ to satisfy \eqref{eq:kernel}. 
Defining  $ t = \left( 1,  \frac{\rho(w) \rho^5(w)}{w} \right) $ 
we see that
\[
   C_\rho(\alpha)\ =\ t^{-1} {^\rho t} 
   \qquad \text{and} \qquad 
   C_{\mbr}(\alpha)\ =\ t^{-1}{^{\mbr} t}\,.
\]
That is, if $ \alpha $ gives the trivial element in $ \brauer{E|L} $
then the cocycle $ C(\alpha) $ represents the trivial class in cohomology. 
Hence the map 
\[
s \colon [\alpha]_{E|L} \in \ker(\eta) \longmapsto [C(\alpha)] \in H^1(D_{12},\frB\otimes \units{K})
\]
is a well-defined homomorphism.

By definition, the homorphism $ j^1 \colon H^1(D_{12},\frB\otimes \units{K}) \to H^1(D_{12}, \frC\otimes \units{K}) $ sends 
$ [C(\alpha)] $ to $ [ jC(\alpha)]$, where $ jC(\alpha) $ is the cocycle determined by
\[
jC_\rho(\alpha) = \left( \frac{\rho(\alpha)}{b}, 1, \frac{b}{\rho(\alpha)} \right) \quad\text{and}\quad
jC_\mbr(\alpha) = \left( 1, \frac{\alpha}{\rho(b)}, \frac{\alpha}{\rho(b)} \right) .
\]  
A calculation shows that this cocyle is equivalent to the cocycle $ \hat C $
determined by 
\[
\hat C_\rho = \left( \alpha, 1, 1 ) \right) \quad\text{and}\quad \hat C_\mbr = (1, \rho(\alpha), \rho^2(\alpha) ) .
\]
Under the identification $ H^1(D_{12}, \frC\otimes \units{L}) \equiv \brauer{E|L} $ the cocyle
$ [\hat C ] $ is sent to $ [\alpha]_{E|L}$. 
In other words $ j^1 \circ s \colon \ker{(\eta)}\to \ker{(\eta)} $ is the 
identity, which shows that $s$ splits the short exact sequence~\eqref{eq:ses-1}.
\smallskip

 \noindent{\bf Step 3: The remaining cases.} \ 
 
 When $ \calG = C_6 $, the proof proceeds exactly as when 
$ \calG = D_{12} $, since in both cases  all norm maps involved can be written using only
 elements in $ C_6 $. 
The only difference here is that $ L = K $.
 \smallskip
 
When $ \calG = D'_6 $ or $ \calG = C_3 $ one has $ H^1(\calG, \frC
 \otimes \units{K} ) = 1 $ and hence $ \ker{(\pi^1)} = 1 $.  
See Table \ref{table:subgps}.  On the other hand, the map $ \pi^0 $ coincides with $ N_{E/k} \colon \units{E} \to \units{k} $ in both cases, once one observes that $ E = K $ when $ \calG = C_3 $.  It follows from \eqref{Eq:H^2} that $ \coker(\pi^0)$ coincides with $ \brauer{k|E} $. 
 \smallskip
  
  When $ \calG = D_6 $ we have, as above, that 
  $ \pi^1 $ coincides with $ \eta \colon \brauer{E|K} \to \brauer{k|F} $.
 We claim that $ \eta $ is injective. 
Indeed, suppose $ \alpha \in \units{E} $ satisfies $ \eta( [ \alpha ]_{E|K}] ) = [1]_{k|F} $, in other words  $ \alpha \rho^2(\alpha) \rho^4(\alpha) = b \rho \mbr(b) $, with $ \alpha = \mbr \rho^3 (\alpha) $ and $ b = \rho^2(b) $.
  Therefore,
  \[ \alpha = \frac{b}{\rho^2(\alpha) } \frac{\rho \mbr (b) }{\rho^4(\alpha) } =  
  \frac{b}{\rho^2(\alpha) } \frac{\mbr\rho^5 (b) }{\rho^4( \mbr \rho^3(\alpha)) } =
  \frac{b}{\rho^2(\alpha) } \mbr\rho^3 \left( \frac{b }{\rho^2(\alpha) } \right) = N_{K/E}\left( \frac{b}{\rho^2(\alpha) }\right).
  \]
 It follows that $ [\alpha]_{E|K} = [1]_{E|K} $, and hence $ H^1(D_6, \frB\otimes \units{K}) = \frac{\brauer{F|K} }{\gamma(\brauer{k|E} )}.$
\end{proof}


\subsubsection{Subgroups of $D_8$}
The remaining Galois groups are subgroups of $D_8$, which is realized as a 
subgroup of $\GL(2,\Z)$ via the map $\DeCo{\varphi}\colon(\rho,\mbr)\mapsto(B,J)$.
Let $D_4\subset D_8$ be the normal subgroup of $D_8$ generated by $\mbr\rho$ and $\rho^2$.
Let $\frA$ be $\Z$ equipped with a $\Z[D_4]$ action where $\mbr\rho$ acts trivially and
$\rho^2$ acts by $-1$.
Define
 \begin{equation}  \label{eq:VD8}
   \frC\ :=\ \Ind_{D_4}^{D_8} ( \frA )\ =\ \Z[D_8]\otimes_{\Z[D_4]} \frA\, .
 \end{equation}
With respect to the basis $\DeCo{e_1} := 1\otimes 1$ and $\DeCo{e_2}:=\rho \otimes 1$,
the matrices representing $\rho$ and $\mbr$ are $B$ and $J$, respectively, and hence
$\frC$ is the $\Z[D_8]$-module structure on $\Z^2$.
For a subgroup $\calG$ of $D_8$, let $\varphi$ be the restriction to $\calG$ of the map
$\varphi\colon D_8\to\GL(2,\Z)$.

\begin{theorem}\label{thm:D8}
 The remaining Galois cohomology groups are given below.
\[
  \begin{tabular}{l||l}
   \cline{1-2}
   %
   \rowcolor[gray]{.8}
   \Blue{ $\calG$ } &
   \Blue{$ H^1(\calG,\scrT_{\varphi}(K)) $ }  \\ \hline\hline
   \rowcolor[rgb]{0.96,0.96,0.9}\rule{0pt}{13pt}
   $D_8 = \langle B, J \rangle $ &  $\brauer{K^{D_4} | K^{D_2}}$ \\
   \rowcolor[gray]{.8}\rule{0pt}{13pt}
   $D_4 = \langle -I,  C \mbr\rho \rangle  $ &  $\brauer{k|K^{D_2}}  \oplus \brauer{k|K^{D_2}}$ \\
   \rowcolor[rgb]{0.96,0.96,0.9}\rule{0pt}{13pt}
   $D'_4 = \langle -I, J  \rangle$ &  $\brauer{K^{C_2}|K}$  \\
   \rowcolor[gray]{.8}\rule{0pt}{13pt}
   $C_4 = \langle B  \rangle  $ & $ \brauer{K^{C_2}|K}$ \\
   \hline
 \end{tabular}
  \qquad
 \begin{tabular}{l||l}
   \cline{1-2}
   %
   \rowcolor[gray]{.8}
   \Blue{ $\calG$ } &
   \Blue{$ H^1(\calG,\scrT_{\varphi}(K)) $ }  \\ \hline\hline
   \rowcolor[rgb]{0.96,0.96,0.9}
   $D_2 = \langle C \rangle$ & $\brauer{k|K} $  \\
   \rowcolor[gray]{.8}
   $D'_2 = \langle J \rangle  $      &  $ 1 $  \\
   \rowcolor[rgb]{0.96,0.96,0.9}
   $C_2 = \langle -I \rangle  $ & $ \brauer{k | K} \oplus \brauer{k | K}$ \\
   \rowcolor[gray]{.8}
   $C_1 = \langle I \rangle $ &  $1$ \\
   \hline
 \end{tabular}
\]

\end{theorem}

\begin{proof}
 We freely use \eqref{Eqn:periodicity}. 

 When $\calG=D_8$, note that 
 $H^1(D_8,\scrT_{\varphi}(K))=H^1(D_8, \Ind^{D_8}_{D_4}(\frA)\otimes\units{K})=
 H^1(D_4, \frA \otimes \units{K})$, by Shapiro's Lemma.
 The restriction of $\frA$ to $D_2 = \langle\mbr \rho \rangle\lhd D_4$ is the
 trivial $\Z[D_2]$-module $\Z$.
 Similar arguments as in the proof of Theorem~\ref{thm:D12} imply that 
 \begin{align*}
  H^1(D_4,\frA \otimes \units{K}) &\ \cong\  H^1(D_4/D_2,(\frA \otimes \units{K})^{D_2})\\
  &\  =\   H^1(\Gal(K^{D_2}/K^{D_4}), \xi \otimes \units{(K^{D_2})}) 
   \ =\ \brauer{K^{D_4}|K^{D_2}}
 \end{align*}
where $\xi$ is the alternating $\Z[\Z/2]$-module.
The last equality is by \eqref{Eqn:periodicity} 
and the identification of 
second cohomology with the Brauer group.

When $\calG=D_4 = \langle\rho^2, \mbr \rho\rangle$, note that
$\Res^{D_8}_{D_4}(\frC)=\Res^{D_8}_{D_4}( \Ind_{D_4}^{D_8}( \frA)) = \frA \oplus \frA$.
Thus we need only compute $H^1(D_4, \frA\otimes \units{K})$, which 
is $\Br(K^{D_4}|K^{D_2})=\Br(k|K^{D_2})$.
\smallskip


When $\calG=D'_4$, we have $D_8= D'_4 \cdot D_4$, and $D'_4\cap D_4=C_2$, 
so the double coset formula implies that 
$\Res^{D_8}_{D'_4}(\frC) = \Res^{D_8}_{D'_4}\bigl( \Ind^{D_8}_{D_4} (\frA)\bigr) =
\Ind^{D'_4}_{C_2}\bigl( \Res^{D_4}_{C_2}(\frA) \bigr)$. 
Therefore,
 \begin{align*}
  H^1(D'_4, {_\varphi}\scrT(K)) &\ =\ H^1(D'_4, \Ind^{D'_4}_{C_2}\bigl(
  \Res^{D_4}_{C_2}(\frA) \bigr) \otimes \units{K}  )\ \cong\ H^1(C_2, \Res^{D_4}_{C_2}(\frA \otimes \units{K} ) )\,.
 \end{align*}
However, $\Res^{D_4}_{C_2}(\frA) = \xi$, and so 
\[
  H^1(C_2, \Res^{D_4}_{C_2}(\frA \otimes \units{K} ) )\ =\ 
  H^1(C_2, \xi \otimes \units{K})\ \cong\ H^2(C_2, \units{K})\ =\ \brauer{K^{C_2}|K}\,.
\]


When $\calG = C_4 = \langle\rho \rangle$, we 
observe that $D_8 = C_4\cdot D_4$ and $C_4 \cap D_4 = C_2$.
Then the same arguments show that $H^1(C_4,{_\varphi\scrT}(K))=\brauer{K^{C_2}|K}$.
\smallskip

When $\calG=D_2 = \langle \mbr\rho \rangle$, 
we have ${_\varphi}\scrT(K) = \units{K} \oplus \xi \otimes \units{K}$.
Hence Hilbert's Theorem 90 gives $H^1(D_2,{_\varphi}\scrT(K))=\Br(k|K)$.\smallskip

When $\calG=D'_2=\langle\mbr \rangle$, we have 
${_\varphi}\scrT(K) \cong \units{K}  \otimes \Z[D'_2]$ as a $D'_2$-module. 
Therefore,
\[
  H^1(D'_2,{_\varphi}\scrT(K))\ \cong\  
  H^1(D'_2, \units{K}  \otimes  \Z[D'_2] )\ \cong\ H^1( \{ I \} , \units{K})\ =\ 1\,,
\]
where the second isomorphism follows from Shapiro's lemma.
\smallskip

When $\calG=C_2=\langle\rho^2 \rangle$, 
we have ${_\varphi}\scrT(K)\cong\T(K)\otimes \xi = (\units{K} \oplus \units{K})\otimes\xi$.
Then \eqref{Eqn:periodicity} 
and the identification of 
second cohomology with Brauer groups gives
$H^1(\calG,{_\varphi}\scrT(K)) \cong\Br(k | K) \ \oplus\ \Br(k | K)$.\smallskip


When $\calG$ is the trivial group, $H^1(\calG;{_\varphi\scrT}(K)) = 1$.
\end{proof}


\subsection{Twisted forms with a given torus}

Given $ \varphi \colon \calG \to \Aut_\Sigma\subset \GL(2,\Z) $, our calculations 
show that the cohomology group  $ H^1(\calG, {_\varphi \scrT}(K) ) $ depends only on the
conjugacy class in $ \GL(2,\Z) $  of the subgroup $ \varphi(\calG) $. 
It is the quotient set  
 \begin{equation}\label{Eq:quotient}
    \frac{H^1(\calG, {_\varphi \scrT}(K) ) }{H^0(\calG, {_\varphi \Aut_\Sigma}) }
 \end{equation}
 that exhibits a dependency on the fan $ \Sigma$, as described in Theorem~\ref{Th:splitting}.

The group $  H^0(\calG, {_\varphi \Aut_\Sigma})$ is  the centralizer of the image $
\varphi(\calG) $ in $ \Aut_\Sigma $, denoted here by $ \sfC_\Sigma(\calG) $, once we identify $
\calG $ and $ \Aut_\Sigma $ with one of the groups in Table \ref{table:subgroups}. 
For simplicity, let $ \sfC(\calG) $ denote the centralizer of $ \varphi(\calG) $ in $ \GL(2,\Z) $.

Although we will not write a complete list with the 
quotient~\eqref{Eq:quotient} displayed for all possible pairs $ \calG \subset \Aut_\Sigma
\subset \GL(2,\Z)$, we will discuss the case when $\calG$ has order $2$ and leave to the reader
the task of calculating the remaining cases.

Suppose that $ \calG = \Gal(K/k) $ has order $2 $.  A conjugacy class of homomorphisms 
$ \varphi \colon \calG  \to \Aut_\Sigma $  is determined by a conjugacy class in $ \Aut_\Sigma $ of (possibly trivial) involutions 
$ \sigma \in \Aut_\Sigma .$ 
Writing \DeCo{$\Aut_\Sigma(2)$} for the set of conjugacy classes of involutions in
$\Aut_\Sigma$, we have 
\[
H^1(K/k,\TAutS) \ =\  
   \coprod_{\sigma \in \Aut_\Sigma(2)} H^1(\calG , {_\sigma}\scrT(K)) / \sfC_\Sigma(\sigma)\, .
\]

Up to conjugacy in $\GL(2,\Z)$, there are four involutions, namely
$I$, $J$, $C$, and $-I$, corresponding to the subgroups $C_1$, $D'_2$, $D_2$, and $C_2$ of
$\GL(2,\Z)$ (see Table~\ref{table:subgroups}).
Then Theorem~\ref{thm:D8} gives
 \[ 
   H^1(\calG, {_\sigma}\scrT(K)) \cong
    \begin{cases}
       1                  &  \text{if $\sigma \sim I$}, \\
       1                  & \text{if $\sigma \sim J$}, \\
     \brauer{k|K}             & \text{if $\sigma \sim C$}, \\
     \brauer{k|K} \oplus \brauer{k|K}   & \text{if $\sigma \sim -I$}, 
    \end{cases}
 \]
where $\sim$ denotes conjugacy.

The action of $\sfC_\Sigma(\sigma)$ is obviously trivial in the first two cases, and we claim
that the action in the third case is trivial, as well. 
In fact all of $\sfC(\sigma)$ acts trivially on $ H^1(\calG, {_\sigma \scrT}(K) ) $  when $\sigma \sim C$.

To see this, denote a cocyle $ \bc $ representing 
$[\bc] \in H^1(\calG, {_\sigma}\scrT(K))$ by its value  
$ \bc_g = (a, b) \in \units{K}\times \units{K} $ on the
non-identity element $g \in \calG$. 
The cocycle condition for $\bc_g = (a, b) $ is 
$ (1,1) = (a,b) {^g (a,b) } = (a,b) \left( g(a), \frac{1}{g(b)} \right) $, i.e.,
$ a\, g(a) = 1$ and $ b = g(b)$. It follows from Hilbert's Theorem 90 that $ a = \frac{u}{g(u) } $ for some $ u \in \units{K} $ and one can write any cocyle in the form $ \bc_g = \left( \frac{u}{g(u)}, b \right) $, with $ u \in \units{K} $ and $ b \in \units{k} $.
The cocycle $ \bc_g $  is equivalent to 
\[ 
t^{-1} \bc_g {^g t}\ =\
\left( \frac{1}{x}, \frac{1}{y} \right) \left( \frac{u}{g(u)},b\right) \left( g(x), \frac{1}{g(y)} \right) = 
\left( \frac{g(x)}{x}\frac{u}{g(u)}, \frac{b}{y\, g(y)} \right) 
\] 
for any $ t =( x, y) \in \units{K} \times \units{K} $. It follows that, by choosing $ x = u $ one can always find an equivalent cocycle of the form $ (1, \alpha ) $, with $ \alpha \in \units{k} $.  In particular, the explicit isomorphism
$ H^1(\calG, {_\sigma}\scrT(\C)) \cong \brauer{k|K} $ is given by sending an element $ [\alpha]_{k|K} $ to the class of the cocyle $ \bc_g = (1,\alpha) $. 
Observe that $ \sfC(C) = \{ \pm I , \pm C \} $. The element $ -I  $ sends $ \bc_g = (1, \alpha) $ to $ {^{-I} \bc}_g = \left( 1, \frac{1}{\alpha} \right) $, and taking $ t = (1, \frac{1}{\alpha}) $ one sees that 
$ {^{-I} \bc}_g =  t^{-1} \bc_g {^g t} $. Therefore, $ -I $ acts trivially on cohomology.  Similarly, $ {^C\bc}_g = (1,\frac{1}{\alpha} ) $ and the previous argument  applies, thus showing that $ \sfC(C) $ acts trivially on 
$ H^1(\calG, {_\sigma \scrT}(K))$.

In the fourth case, with $\sigma = -I$, Theorem~\ref{thm:D8} gives $H^1(\calG, {_\sigma}\scrT(K)) = \brauer{k|K} \oplus \brauer{k|K} $. Explicitly, it is easy to see that the cocycle condition for 
$ \bc_g = (a,b) $ is precisely $ (a, b)  \in (\units{k}, \units{k} ) $, whereas two cocycles $ (a, b) $ and $ (a', b') $ are equivalent if and only if $ a' = a\, x g(x) = a\, N_{K/k}(x) $ and $ b' = b\, y g(y) = b\, N_{K/k}(y) $, for some $ x, y \in \units{K} $. 

It follows that $ \GL(2,\Z) $ acts on $ \brauer{k|K} \oplus \brauer{k|K} $ via its usual action on $\units{k} \times \units{k}$ and, since   $ \brauer{k|K} $ is a $ \Z/2\Z $ vector space, one concludes that this action descends to an action of $ \GL(2,\Z/2\Z) $.
Therefore $ -I = A^3 = B^2 $ and $ C $ act trivially, and the actions of $ J $ and $ B$ coincide: $ {^J[a,b]} = {^B[a,b]} = [b,a] $.
Furthermore, $ {^A[a,b]} = [b,ab] $ and $ {^{A^2}}[a,b] = [ab,a] $.

\begin{rem}
\label{rem:actions}
We identify two basic actions on $ \brauer{k|K}\oplus \brauer{k|K}$: a $ \Z/2\Z $-action interchanging coordinates, $ [a,b] \mapsto [b,a] $, and a $ \Z/ 3\Z $-action with orbits of the form 
\begin{equation}
\label{eq:action3}
 [a,b] \mapsto [b,ab] \mapsto [ab, a ] \mapsto [a, a^2b ] \equiv [a,b].
 \end{equation}
The calculations above show that the action of a group in Table \ref{table:subgroups} that contains $ -I $ on $ H^1(\calG,{_\varphi \scrT}(K) ) $ is either trivial or descends to the actions of $ \Z/2\Z $ or $ \Z/3\Z$ described above.
Note that the quotient set
$ \left( \brauer{k|K}\oplus \brauer{k|K}\right)/\left( \Z/2\Z \right)  $ is the set-theoretic $ 2 $-fold symmetric product $ SP_2(\brauer{k|K}) $ of $ \brauer{k|K} $.
\hfill\qed
\end{rem}

\begin{example}
The case of real toric surfaces follows immediately from the discussion above.
In this case $ \brauer{\R} = \brauer{\R|\C} = \{ [1], [-1]\} $. 
One sees from \eqref{eq:action3} that for
$ \Aut_\Sigma = D_{12}, D_6$ there are two
orbits $ \{ [1,1] \} $ and $ \{ [1,-1], [-1,1], [-1,-1] \}$.  
Similarly, for $ \Aut_\sigma = D_8, D'_4, C_4 $ 
there are three orbits
$ \{ [1,1] \} $, $ \{ [1,-1], [-1,1] \} $ and $ \{ [-1,-1] \} $. 

This should be compared to Theorem~5.3.1 of Delaunay's Thesis~\cite{D04}.
Note that Delaunay's types I, IV correspond to our $\sigma = I, -I$,
while her types II, III correspond to $\sigma$ a reflection, $C$ or $J$;
the distinction between her types II and III is not between $C$ and $J$, rather whether the
reflection $\sigma$ fixes a two-dimensional cone of $\Sigma$.
\hfill\qed
\end{example}

We conclude by summarizing the case of Galois groups of order $ 2 $ in Table~\ref{Table: surfaces over quadratic extension}, whose last column contains the sizes of 
$ H^1( \calG, {_\sigma \scrT}(K) )/ H^0(\calG, \Aut_\Sigma) $  for $ K/k = \C/\R $.

\begin{table}[!ht]
\caption{Classification of arithmetic toric surfaces over a quadratic field extension.}
\label{Table: surfaces over quadratic extension}
\begin{tabular}{l||l|l|c}
\rowcolor[gray]{.8}
\Blue{$ \sigma $}  & \Blue{$ \Aut_\Sigma $} & 
\Blue{$ H^1( \calG, {_\sigma \scrT}(K) )/ H^0(\calG, \Aut_\Sigma) $} & 
\Blue{$ | H^1( \C/\R , {_\sigma \scrT}) / \sfC_\Sigma(\sigma) | $} \\ \hline\hline
\rowcolor[rgb]{0.96,0.96,0.9}
$ I $ &  any  & $1 $ & $ 1 $ \\
\rowcolor[gray]{.8}
$ J $ & $ D_{12}, D_8, D'_6, D'_4, D'_2 $ & 1 & $ 1 $\\
\rowcolor[rgb]{0.96,0.96,0.9}
$ C $ & $ D_8, D_4, D_2 $  & $ \brauer{k|K} $ & $ 2 $ \\
\rowcolor[gray]{.8}
$ -I $ & $ D_{12}, D_6 $ & $ \left( \brauer{k|K} \oplus \brauer{k|K} \right)/(\Z/3\Z) $ & $ 2 $  \\
\rowcolor[rgb]{0.96,0.96,0.9}
$ -I $ & $ D_{8}, D'_4, C_4 $ & $ SP_2(\brauer{k|K}) $ & $ 3 $ \\
\rowcolor[gray]{.8}
$ -I $ & $ D_4, D_2 $ & $ \brauer{k|K} \oplus \brauer{k|K} $ & $ 4 $ 
\end{tabular}
\end{table}

\providecommand{\bysame}{\leavevmode\hbox to3em{\hrulefill}\thinspace}
\providecommand{\MR}{\relax\ifhmode\unskip\space\fi MR }
\providecommand{\MRhref}[2]{%
  \href{http://www.ams.org/mathscinet-getitem?mr=#1}{#2}
}
\providecommand{\href}[2]{#2}

\end{document}